\documentclass[onefignum,onetabnum]{siamart171218}

\usepackage{lipsum}
\usepackage{amsfonts}
\usepackage{graphicx}
\usepackage{epstopdf}
\usepackage{algorithmic}
\ifpdf
  \DeclareGraphicsExtensions{.eps,.pdf,.png,.jpg}
\else
  \DeclareGraphicsExtensions{.eps}
\fi

\newsiamremark{remark}{Remark}
\newsiamremark{hypothesis}{Hypothesis}
\crefname{hypothesis}{Hypothesis}{Hypotheses}
\newsiamthm{claim}{Claim}

\headers{Energy-based transport for amortized bayesian inference
}{R. Baptista, H. Kaveh, And A. Stuart}
\hypersetup{
  pdftitle={Energy-based transport for amortized bayesian inference
  },
  pdfauthor={H. Kaveh, R. Baptista, And A. Stuart}
}
\title{
Energy-based transport for amortized bayesian inference\thanks{The authors are listed in alphabetical order.}
}
\author{Ricardo Baptista\thanks{Department of Statistical Sciences, University of Toronto 
  (\email{r.baptista@utoronto.ca}).}
  \and
Hojjat Kaveh\thanks{
Mechanical and Civil Engineering, California Institute of Technology
(\email{hkaveh@caltech.edu}). Corresponding author.
}
\and Andrew M. Stuart\thanks{Computing and Mathematical Sciences, California Institute of Technology}
(\email{astuart@caltech.edu})}

\usepackage{amsopn}

\makeatletter
\newcommand*{\addFileDependency}[1]{
  \typeout{(#1)}
  \@addtofilelist{#1}
  \IfFileExists{#1}{}{\typeout{No file #1.}}
}
\makeatother

\usepackage{macros}
\usepackage{graphicx}
\ifpdf
\hypersetup{
  pdftitle={Amortized Energy Distance-Based Operator Learning for Infinite-Dimensional Bayesian Inverse Problems},
  pdfauthor={R. Baptista, H. Kaveh, And A. Stuart}
}
\usepackage{xcolor}

\begin{document}

\maketitle
\begin{abstract} We consider amortized Bayesian inference for nonlinear inverse problems using only samples from the joint distribution of parameters and observations, including problems with unknown functions in a Banach space. Classical methods such as Markov chain Monte Carlo solve a new inference problem for each observation, making repeated posterior inference computationally prohibitive, particularly in infinite dimensions. Amortized Bayesian inversion instead learns a reusable map that rapidly generates posterior samples for new observations. We learn an observation-dependent transport map that pushes a reference measure to an approximate posterior. Training minimizes the average energy distance between the posterior and the learned pushforward. Averaging over observations allows generalization across observation instances and efficient amortized inference. Furthermore, the formulation is likelihood-free, requiring only samples from the joint distribution and avoiding likelihood evaluation. In addition, the use of an energy-distance objective removes the need for invertibility of the transport map and for computation of Jacobian determinants, enabling flexible parameterizations in high- and infinite-dimensional settings. Moreover, when the posterior has a density with respect to a Gaussian prior measure, we construct transport maps as the identity plus a learnable map valued in the prior’s Cameron--Martin space. This guarantees that the learned posterior remains absolutely continuous with respect to the prior. In infinite dimensions, the transport map is parameterized using neural operators, enabling use at different grid resolutions. We demonstrate the approach on a finite-dimensional problem and PDE-based porous-medium flow and seismic inverse problems. The learned transport captures multimodality and dominant posterior modes while enabling fast sampling.
\end{abstract}


\setcounter{section}{0}

\section{Introduction}

Inverse problems arise in many areas of science and engineering 
\cite{stuart_inverse_2010,dashti_bayesian_2017,tarantola_inverse_2005}, 
where one seeks to infer unknown parameter \(u \in \mathcal{U}\) 
from indirect and noisy observation \(y \in \mathcal{Y}\). 
Examples include the estimation of unknown fields and model parameters in 
mechanics, medical imaging, and geophysics, where the available 
data are typically sparse, noisy, and indirectly related to the quantities of 
interest 
\cite{kaveh_induced_2023,mousavi_bayesian_2024,arridge_solving_2019,kaveh_data_2026,kaveh_bayesian_2026}. 
These challenges motivate probabilistic formulations of inverse problems that quantify the uncertainty 
in the inferred parameters rather than producing a single deterministic estimate.

A common mathematical formulation is based on the stochastic forward model:
\begin{equation}
\label{eq:GIP}
    y = G(u,\eta),
\end{equation}
where \(G:\mathcal{U} \times \mathcal{E}\rightarrow\mathcal{Y}\) denotes the forward operator, mapping parameter $u$ and noise \(\eta\) to the observation.
In the Bayesian framework, the unknown parameter \(u\) is 
modeled as a random variable with prior distribution \(\rho\) on the parameter space $\mathcal{U}$. The forward model 
and the noise then induce a joint distribution \(\gamma\) on \((u,y)\), 
with marginal distribution \(\kappa\) on the observation space $\mathcal{Y}$. For a realized 
observation \(y^\dagger\), uncertainty in \(u\) is described by the posterior 
distribution \(\pi(\cdot \mid y^\dagger)\). 
Given a likelihood function \(\ell(y \mid u)\) that is defined by the conditional distribution for the observation $y$ under the forward model~\eqref{eq:GIP}, the posterior is defined by
Bayes' rule as 
\[
    \pi(u \mid y^\dagger)
    =
    \frac{\ell(y^\dagger \mid u)\rho(u)}
    {\kappa(y^\dagger)},
\]
In infinite-dimensional settings, Bayes' rule is interpreted through the 
Radon--Nikodym derivative of the posterior with respect to the prior.

In many applications, posterior inference must be performed not for a single 
observation, but repeatedly for many different realizations of the data. This 
situation arises, for example, when analyzing many experimental configurations in experimental design~\cite{huan_optimal_2024} or many synthetic and observational data sets in 
uncertainty quantification, design, and real-time inference \cite{radev_bayesflow_2020,karumuri_learning_2024}. Classical sampling methods such as Markov chain Monte Carlo (MCMC) \cite{brooks2011handbook}, and infinite
dimensional variants \cite{cotter_mcmc_2013}, require the solution of a new inverse 
problem for each observation \(y^\dagger\). This repeated online cost can become computationally prohibitive. 

This motivates amortized Bayesian inference. The central idea is to replace 
repeated, observation-specific inference with a reusable model learned during an 
offline training stage, which may be applied with any observation. Once trained, 
this model can be evaluated rapidly for any observation \(y^\dagger\), producing approximate posterior samples without evaluation of the likelihood multiple times as would be
required to deploy MCMC. Thus, amortization is particularly attractive when inference must be repeated many times for different observations.

In this work, we pursue an amortized inference approach based on conditional transport 
maps. In particular, we seek an observation-dependent \emph{transport map} \(T_\theta(\cdot \,;y) \colon \mathcal{Z} \to \mathcal{U}\) that pushes forward a reference measure \(\muref\) on a latent space $\mathcal{Z}$ to an approximation of the 
posterior distribution. That is, we choose parameter $\theta$ such that, for $y$ drawn from 
the marginal $\kappa$,
\begin{equation*}
    T_\theta(\cdot \,;y)_\# \muref
    \approx
    \pi(\cdot \mid y).
\end{equation*}
If the map satisfies the push-forward condition exactly, then $T_\theta(z;y) \sim \pi(\cdot \mid y)$ for $z \sim \muref$. Therefore, posterior samples for a new observation 
\(y^\dagger\) can be obtained by drawing samples from the reference measure and evaluating the 
learned map \(T_\theta(\cdot \,;y^\dagger)\) on those samples.

Measure transport provides a natural framework for representing complicated 
probability distributions as pushforwards of simpler reference measures 
\cite{villani_optimal_2009,santambrogio_optimal_2015, marzouk_introduction_2016}. In Bayesian inference, 
transport maps have been used to approximate posterior distributions by 
constructing deterministic couplings between a tractable reference measure and 
the target posterior \cite{moselhy_bayesian_2012}. 
Many existing approaches are based on objectives involving the 
Kullback--Leibler divergence \cite{baptista2024representation, cao_lazydino_2026}. However, such formulations often require the map 
to be invertible and require evaluation of Jacobian determinants. These 
requirements are restrictive in high-dimensional problems and become especially 
problematic in infinite-dimensional settings. Alternative approaches seek transports by minimizing optimal transport distances~\cite{baptista2024conditional, baptista2025conditional, taghvaei2022optimal}, often requiring adversarial optimization, which is challenging to scale to high-dimensional settings.

We address the limitations of existing
transport-based methodologies for amortized inference 
by training the transport map using an averaged 
energy-distance objective between the true posterior and the learned pushforward 
distribution. Specifically, for an observation \(y\sim\kappa\), we approximate 
the posterior \(\pi(\cdot\mid y)\) by the pushforward distribution 
\(B_\theta(\cdot \,;y) :=T_\theta(\cdot \,;y)_\#\muref\), and minimize
the objective function
\begin{equation}
\label{eq:objective_e}
    \LossL(\theta)
    =
    \mathbb{E}^{y\sim\kappa}
    \left[
        \DE^2\bigl(\pi(\cdot\mid y), B_\theta(\cdot \,;y)\bigr)
    \right].
\end{equation}
Here \(\DE\) denotes the energy distance that is defined for 
$\mu, \nu \in \mathcal P(\mathcal U)$, two probability measures on \(\mathcal{U}\), by
\begin{equation}
\label{eq:ED_def}
\DE^2(\mu,\nu)
:=
2\mathbb{E}^{(u,v)\sim \mu\otimes\nu}\|u-v\|
-
\mathbb{E}^{(u,u')\sim \mu\otimes\mu}\|u-u'\|
-
\mathbb{E}^{(v,v')\sim \nu\otimes\nu}\|v-v'\|.
\end{equation}
For a  probability measure \(\nu \in \mathcal P(\mathcal U)\) and 
\(u \in \mathcal U\), the energy score \cite{bach_machine_2025} is defined as
\begin{equation}
\label{eq:ES_def}
\ES(\nu,u)
=
\mathbb{E}^{v\sim \nu}\|u-v\|
-
\frac12\,
\mathbb{E}^{(v,v')\sim \nu\otimes \nu}\|v-v'\|.
\end{equation}
We will show that the objective defined in Eq.~\eqref{eq:objective_e} is equivalent, up to terms independent of \(\theta\), to an averaged energy-score; this enables
purely sample-based evaluation of the desired objective function. Chapter 5 of \cite{bach_machine_2025} lays out a general methodology for amortization in the solution of inverse problems using strictly proper scoring rules. Our energy distance objective is a specific example of the scoring rule methodology.

The energy distance defines a discrepancy between probability measures and is 
particularly amenable to settings in which the measures are available only 
through samples. Moreover, it does not require density evaluation, 
invertibility of the transport map, or computation of Jacobian determinants. We show that the 
resulting objective can be written entirely in terms of expectations over the joint 
distribution \(\gamma\) and its marginals. Consequently, the method is 
likelihood-free: it only requires samples from the joint distribution of 
parameters and observations, rather than explicit access to the likelihood function or the 
posterior density. Our likelihood-free formulation is related to Sequential Neural Likelihood (SNL) methods, which also use simulator-generated samples to avoid explicit likelihood evaluation, but differs in that our method learns an observation-amortized posterior-generating transport map, whereas SNL learns a neural surrogate for the likelihood \cite{papamakarios_sequential_2019}. Flow-matching posterior estimation also trains an amortized posterior sampler from joint simulator samples \cite{wildberger2023flow}, but it represents the posterior through an ODE flow and learns the corresponding conditional velocity field. In contrast, our approach trains a direct pushforward map by minimizing an energy-distance discrepancy between posterior measures, without requiring invertibility, density evaluation, Jacobian determinants, or divergence computations.

When the parameter \(u\) is a function, this transport-map perspective must also 
respect the measure-theoretic structure of Bayesian inverse problems in function 
space. In typical Bayesian inverse problems, the posterior measure is absolutely 
continuous with respect to the prior \cite{stuart_inverse_2010};
in particular, the posterior and prior cannot be mutually singular. 
In infinite-dimensional problems, it is a delicate issue to avoid such mutual singularity.
For example, Gaussian measures are mutually singular unless the conditions of
the Feldman-H\/ajek theorem are met  \cite[Theorem 37]{dashti_bayesian_2017}. 
We address this issue in the context of inverse problems in Banach space, with
Gaussian priors. In particular, we construct transport maps that respect the
measure-theoretic structure of Gaussian priors on function spaces \cite{bogachev_gaussian_1998,bogachev_triangular_2005}. Motivated by
the Cameron--Martin theorem \cite[Theorem 32]{dashti_bayesian_2017}, 
we choose the reference measure to be the prior distribution on $\mathcal{U}$ and write the map $T_\theta(\cdot \,;y) \colon \mathcal{U} \rightarrow \mathcal{U}$ as an identity perturbation
\[
    T_\theta(u;y) = u + {H}_\theta(u;y),
\]
where the architecture is chosen so that the perturbation
\({H}_\theta(u;y)\) lies in the Cameron--Martin space associated with the
prior. In our work, the map \({H}_\theta\) is represented using a neural operator
\cite{li_fourier_2021}, with its output constrained through the covariance
structure of the Gaussian prior. This construction builds on the
measure-theoretic underpinnings of Gaussian measures in Banach space and is designed to ensure that the
learned posterior $B_\theta(\cdot\,; y)$ is absolutely 
continuous with respect to the prior measure $\rho$.

The main contributions of this paper are as follows:
\begin{itemize}
    \item[(C1)] We formulate an amortized approach to Bayesian inverse problems based 
    on minimizing an averaged energy-distance objective between the true posterior 
    and a learned pushforward distribution. This avoids the need for 
    invertibility of the transport map and the computation of 
    Jacobian determinants.

    \item[(C2)] We show that the objective can be expressed using only samples from the 
    joint distribution of parameters and observations and its marginals. The 
    resulting method is therefore likelihood-free and is applicable in 
    simulation-based inverse problems where the likelihood is unavailable or 
    difficult to evaluate.

    \item[(C3)] We design transport maps for infinite-dimensional Bayesian inverse 
    problems by parameterizing them as identity perturbations in the 
    Cameron--Martin space of the Gaussian prior. This construction preserves 
    the measure-theoretic structure of the posterior and is implemented using 
    neural operators.

    \item[(C4)] We demonstrate the method on a finite-dimensional nonlinear inverse 
    problem and on two PDE-constrained inverse problems: a Darcy flow inverse 
    problem and a wave-equation inverse problem. The numerical results show 
    that the learned transport maps capture posterior structure, including 
    multimodality and dominant KL modes, while enabling fast posterior sampling 
    for new observations.
\end{itemize}

The remainder of the paper is organized as follows. 
Section~\ref{sec:methodology} develops the methodology corresponding to 
Contributions (C1)--(C3). In particular, 
Subsection~\ref{ssec:objective} introduces the averaged energy-distance 
objective in (C1), derives its sample-based likelihood-free form in (C2), and 
relates the objective to the energy score. 
Subsection~\ref{ssec:transportmap} describes the transport-map architecture and 
the Cameron--Martin-informed parameterization used for infinite-dimensional 
Bayesian inverse problems, corresponding to (C3). 
Section~\ref{sec:numerics} addresses (C4) by presenting numerical experiments 
on a finite-dimensional nonlinear inverse problem, a Darcy flow inverse problem, 
and a wave-equation inverse problem. 
Section~\ref{sec:conclusion} summarizes the main findings, discusses the 
computational advantages of the proposed amortized approach, and outlines open 
directions. The paper ends with four appendices. Appendix~\ref{sec:preliminaries} reviews 
the Gaussian measure and Cameron--Martin space background needed for the 
construction used in (C3). Appendices~\ref{app:proof_objective_e} and 
\ref{app:proof_ES} contain proofs of technical lemmas underlying 
Subsection~\ref{ssec:objective}. Appendix~\ref{sec:numerical_details} discusses additional details of the numerical experiments, including the 
scaling of the methodology with respect to the amount of training data. 

\section{Methodology}
\label{sec:methodology}
This section introduces the proposed amortized transport methodology. 
Section~\ref{ssec:objective} defines the averaged energy-distance objective used to learn 
conditional pushforward approximations of the posterior distributions, derives a 
sample-based form involving only the joint distribution of parameters and 
observations, and relates the objective to the energy score. 
Section~\ref{ssec:transportmap} then describes the transport-map architecture;
in the setting of an unknown parameter in a Hilbert space, the architecture is
further constrained by means of a Cameron--Martin-informed construction that is
natural when Gaussian priors are adopted on the function space.

\subsection{Objective Function}
\label{ssec:objective}
Given the Bayesian inverse problem defined by the model in Eq.~\eqref{eq:GIP}, 
our goal is to approximate the posterior distribution $\pi(u \mid y)$. We
employ a parametric family of distributions $B_\theta(u\,; y)$ for the approximation, aiming to amortize over
$y \sim \kappa$ and choose $\theta$ to optimize this approximation. 
We represent the approximate posterior $B_\theta(u\,; y)$ as the pushforward of a reference measure $\muref$ under a transport map $T_\theta(\cdot \,; y)$. That is,
\[
B_\theta(u \,; y) = T_\theta(\cdot \,; y)_\# \muref(u),
\]
where the variable $\theta$ denotes the parameters of a neural network or neural operator employed to define $T_\theta(\cdot\,; y)$. 

A key aspect of our approach is the use of the energy distance to quantify the discrepancy between distributions, resulting in the objective function defined by Eq.~\eqref{eq:objective_e}. In contrast to approaches based on minimizing the Kullback--Leibler divergence, this choice does not require the transport map to be invertible, nor does it require evaluation of Jacobian determinants, which are both necessary for evaluation of the pushforward density. This allows for flexible map parameterizations and makes the method applicable in high- and infinite-dimensional settings.

Moreover, the averaged energy distance admits a representation in terms of
expectations with respect to distributions that are directly accessible in
practice, namely the joint distribution $\gamma$ and its marginal $\kappa$.
Since the energy distance is naturally compatible with sample-based approximation, this representation enables training using only samples from the
joint distribution, without requiring explicit knowledge of the prior or the
likelihood. The resulting formulation is therefore \emph{likelihood-free} and falls in the category of \emph{simulation-based inference} techniques.

Using the definition of the energy distance in Eq.~\eqref{eq:ED_def}, the following 
lemma rewrites the objective in Eq.~\eqref{eq:objective_e} entirely in terms of expectations over \(\gamma\) 
and its marginal \(\kappa\).
The proof is given in Appendix~\ref{app:proof_objective_e}.

\
\begin{lemma}
\label{lemma:objective_e_updated}
Minimizing the objective function $\LossL(\theta)$ in Eq.~\eqref{eq:objective_e} is equivalent 
to minimizing the objective $\LossJ(\theta)$ where
\label{eq:J(theta)}
\begin{align}
    \LossJ(\theta) &= 2\mathbb{E}^{\left(z, (u^\prime, y^\prime)\right)\sim \left(\muref \otimes \gamma \right)} \| T_\theta(z; y^\prime) - u^\prime\|\notag\\ 
    &\qquad \qquad-\mathbb{E}^{\left(z,z',y''\right)\sim \left(\muref \otimes \muref \otimes \kappa \right)}\| T_\theta(z; y'')-T_\theta(z^\prime; y'') \|.
\end{align}
\end{lemma}


\begin{remark}[Choice of reference measure]
\label{rem:reference_measure}
Throughout this paper, we primarily choose the latent space $\mathcal{Z} = \mathcal{U}$ and take the reference measure 
\(\muref\) to be the prior \(\rho\). However, 
the energy-distance objective does not require this particular choice. For 
example, one could first construct a tractable approximation to the posterior, 
such as a Gaussian approximation obtained by variational inference, and then 
learn a transport map from this reference measure to the true posterior. More 
generally, \(\muref\) need not be defined on \(\mathcal U\); it may be a 
measure on a (possibly finite-dimensional) latent space provided that the range of the transport map takes values in 
\(\mathcal U\). 
\end{remark}

 We now show that the objective in Eq.~\eqref{eq:objective_e} is equivalent, up to 
terms independent of \(\theta\), to an averaged energy-score objective, with the 
energy score defined in Eq.~\eqref{eq:ES_def}. The proof is given in
Appendix~\ref{app:proof_ES}.

\begin{lemma}[Equivalence between averaged energy distance and expected energy score]
\label{lem:ES}
Let $\gamma \in \mathcal P(\mathcal U \times \mathcal Y)$ be the joint law of $(u,y)$, 
with marginal law $\kappa$ on $\mathcal Y$ and conditional law $\pi(\cdot \mid y)$ on $\mathcal U$. 
Let $B_\theta(\cdot \,; y)$ be a family of probability measures on $\mathcal U$. Then
\begin{equation}
\label{eq:ES}
\mathbb{E}^{y \sim \kappa}
\Big[
\DE^2\bigl(\pi(\cdot \mid y),B_\theta(\cdot \,; y)\bigr)
\Big]
=
2\,\mathbb{E}^{(u,y)\sim \gamma}
\Big[
\ES\bigl(B_\theta(\cdot \,; y),u\bigr)
\Big]
- \mathsf{const},
\end{equation}
where
\[
\label{eq:C_ES2}
\mathsf{const}
:=
\mathbb{E}^{y\sim\kappa}
\Big[
\mathbb{E}^{(u,u')\sim \pi(\cdot\mid y)\otimes \pi(\cdot\mid y)}
\|u-u'\|
\Big]
\]
is independent of $\theta$. Consequently, minimizing the expected energy score is equivalent to minimizing the objective defined in Eq.~\eqref{eq:objective_e}.
\end{lemma}

\begin{remark}[Generalized energy distances]
\label{rem:generalized_energy_distance}
In general, one may replace the norm terms in Eq.~\eqref{eq:ED_def} and Eq.~\eqref{eq:ES_def} that define the energy distance by \(\|\cdot\|^\beta\), for suitable choices of 
\(\beta\). This leads to the generalized energy distance and corresponding 
energy score. More generally, one can also replace the norm with a general metric for data living in an arbitrary metric space provided that the resulting distance remains a valid divergence between probability measures~\cite{szekely2013energy}. In this paper, we focus on the energy distance with exponent \(\beta=1\); see~\cite[Section 11.3]{bach_machine_2025} for the properties of energy distance including its derivation from a strictly proper scoring rule and connection to the Continuous Ranked Probability Score, which is commonly used to assess probabilistic forecasts~\cite{gneiting2007strictly}.
\end{remark}

Throughout this paper, we will work under the following assumption about the data used
to train our amortized Bayesian inference model:

\begin{dassumption}
\label{assumption:data}
We are given \(N\) i.i.d.\ samples
\[
    \{(u^{(i)},y^{(i)})\}_{i=1}^N
    \sim \gamma,
\]
where \(\gamma\) is the joint distribution of parameters and observations.
\end{dassumption}

The following proposition provides an empirical approximation of the objective function \(\LossJ(\theta)\).
\begin{proposition}
Under Data Assumption~\ref{assumption:data}, with the reference measure 
\(\muref\) taken to be the prior distribution $\rho$, an empiricial approximation to \(\LossJ(\theta)\) is given by
\begin{equation}
\label{eq:JN}
\begin{aligned}
\LossJ^N(\theta)
&= \frac{2}{N(N-1)} \sum_{\substack{i,j=1\\ i\neq j}}^N
\left\| T_\theta\!\big(u^{(i)}; y^{(j)}\big) - u^{(j)} \right\| \\
&\quad - \frac{1}{N(N-1)(N-2)} \sum_{\substack{i,j,k=1\\ i\neq j,\ i\neq k,\ j\neq k}}^N
\left\| T_\theta\!\big(u^{(i)}; y^{(k)}\big)
      - T_\theta\!\big(u^{(j)}; y^{(k)}\big) \right\|.
\end{aligned}
\end{equation}
The first term encourages $T_\theta(u^{(i)};y^{(j)})$ to match the target $u^{(j)}$, while the second promotes spread in the transported samples for each fixed conditioning value $y^{(k)}$.
\end{proposition}

\subsection{Parameterization of the Transport Map}
\label{ssec:transportmap}
To complete the formulation, we now specify the class of transport maps $T_\theta$.
In finite dimensions, we use standard deep neural network architectures. In the infinite-dimensional setting, 
the transport map is subject to structural constraints, which we incorporate in the architecture as a form of inductive bias. 
Since the posterior is absolutely continuous with respect to the prior, it is desirable that
the approximation preserves this property. Satisfying this property in infinite dimensions requires arbitrary perturbations of $u$ to induce pushforward measures that remain absolutely continuous with respect to the prior. For the Gaussian priors considered in this work, we next describe the structural assumptions under which this property holds. The detailed background on this topic is found in Appendix~\ref{sec:preliminaries}.

Motivated by the Cameron--Martin theorem explained in Appendix~\ref{sec:preliminaries}, we  parameterize the transport map as
\begin{equation}
\label{eq:T_CM}
T_\theta(u; y) = u + C^{1/2}\,S_\theta(u; y),
\end{equation}
where $C$ is the covariance operator of the Gaussian prior. Provided that $S_\theta(\cdot\,, y)$ maps $\mathcal{U}$ into itself for each fixed $y$ in the support
of $\kappa$, this construction ensures that the perturbation of $u$ lies in the Cameron--Martin space. As we will see, this form of inductive bias in the
learned model will help promote the correct behaviour of the computed posterior for components of the posterior that are not sensitive to the observed data; hence, it leads to more accurate posterior inference.

The map $S_\theta$ is implemented using a neural operator, allowing us to represent nonlinear maps between functions in a way that is consistent with the infinite-dimensional structure of the problem. For example, in the PDE-based inverse problem (Experiments 2 and 3), which we will discuss in Section~\ref{sec:numerics}, the latent field $u$ is defined on the spatial domain $(0,1)$ with homogeneous Neumann boundary conditions. The Mat\'{e}rn covariance operator
\[
C = \sigma^2 (-\Delta + \tau^2 I)^{-\alpha},
\]
given variance amplitude $\sigma^2 > 0$, inverse correlation length $\tau > 0$ and smoothness parameter $\alpha > 0$, is diagonalized by the cosine basis, which corresponds to the eigenfunctions of the Laplacian with Neumann boundary conditions. To respect this structure, we parameterize $S_\theta$ using a Fourier neural operator adapted to Neumann boundary conditions, implemented via the discrete cosine transform (DCT). This ensures that both the covariance operator and the neural parameterization share the same spectral structure.

The action of \(C^{1/2}\) in Eq.~\eqref{eq:T_CM} is implemented spectrally. Let  
\(\{(\lambda_k,\phi_k)\}_{k\geq 1}\) denote the eigenpairs of the covariance 
operator \(C\). By writing
\[
S_\theta(u; y) = \sum_{k} s_{\theta,k}(u; y)\,\phi_k,
\]
we can apply \(C^{1/2}\) by scaling each coefficient in the sum as
\[
C^{1/2} S_\theta(u; y) = \sum_{k} \sqrt{\lambda_k}\, s_{\theta,k}(u; y)\,\phi_k.
\]
Since the eigenvalues $\lambda_k$ decay to zero as $k \to \infty$, the operator $C^{1/2}$ attenuates high-frequency modes and thus acts as a smoothing operator. This ensures that the perturbation lies in a more regular subspace, consistent with the Cameron--Martin structure.
For comparison, we also consider numerical experiments using the baseline parameterization
\begin{equation}
T_\theta(u; y) = u + S_\theta(u; y),
\end{equation}
which, hence, does not enforce this Cameron--Martin structure.

\section{Numerical Experiments}
\label{sec:numerics}

This section evaluates the proposed amortized transport methodology on a sequence of inverse problems of increasing complexity. The purpose of these experiments 
is to illustrate the main ideas of the proposed methodology rather than to provide a comprehensive benchmarking study. The problems are not at the edge of computational feasibility, but rather are chosen to be tractable via MCMC methodology in order to demonstrate the benefits of our approach in comparison to a gold-standard method. 
In particular, we show that: (i) use of the energy-distance objective leads to
efficient observation-dependent posterior approximations that outperform MCMC when multiple observation instances are encountered; and that (ii) the Cameron--Martin-informed parameterization is beneficial in function-space settings, analogously to the benefits of function-space formulations of MCMC. 

In Subsection~\ref{sec:exp1}, we consider a finite-dimensional nonlinear example that has a tractable non-Gaussian posterior, followed by two PDE-based inverse problems defined by Darcy flow and wave propagation equations in Subsections~\ref{sec:darcy} and~\ref{sec:waveeq}, respectively. In the Darcy problem, we recover a log-normal permeability field, and in the wave propagation problem, we recover a piecewise constant wavespeed field, defined via thresholding a Gaussian random field. Both examples can be formulated in terms of a Gaussian random field prior, and thus the structure from Subsection~\ref{ssec:transportmap} can be employed. In all three sets of experiments, we work with an inverse problem in which the noise is 
additive and replace the model in~\eqref{eq:GIP} by
\begin{equation} \label{eq:GIPadd}
y = \cG(u) + \eta,
\end{equation}
for some map $\cG: \mathcal{U} \to \mathcal{Y}$. Throughout this section, we work under Data Assumption~\ref{assumption:data}. For the two PDE-based inverse problems, we compare our method with the pCN method \cite{cotter_mcmc_2013} since this is designed for function-space inference problems. Additional details on the numerical experiments are included in Appendix~\ref{sec:numerical_details}.

\subsection{Experiment 1: Tractable Non-Gaussian Posterior} \label{sec:exp1}

We consider a finite-dimensional inverse problem for which the posterior distribution can be accurately approximated using a quadrature rule, which allows us to compute the errors in the solution with respect to the amount of training data and the number of parameters of the model. This example serves two purposes. First, it provides a clear illustration of the amortized inverse problem framework, where a single learned map is used to approximate the posterior distributions corresponding to different observations. Second, despite its low dimensionality, this example is nontrivial: for certain ranges of the observation, the posterior distribution is bimodal, whilst for others it is unimodal; this makes the amortization challenging.

Consider \eqref{eq:GIPadd} with $\mathcal{U}=\mathcal{E}=\mathcal{Y}=\mathbb{R}$ and $\cG(u)=u^2.$
The prior and observational noise are given by
\[
u \sim \mathcal{N}(m_0,\sigma_0^2),
\qquad
\eta \sim \mathcal{N}(0,\sigma^2).
\]
By Bayes' theorem, the posterior density is given by
\[
\pi(u \mid y)
=
\frac{1}{Z(y)}
\exp\!\left(
-\frac{(u - m_0)^2}{2\sigma_0^2}
-\frac{(y - \cG(u))^2}{2\sigma^2}
\right),
\]
where the normalizing constant $Z(y)$ is
\[
Z(y)
=
\int_{\mathbb{R}}
\exp\!\left(
-\frac{(u - m_0)^2}{2\sigma_0^2}
-\frac{(y - \cG(u))^2}{2\sigma^2}
\right)\,du.
\]

In the numerical experiments below, we consider the specific parameter values
\[
m_0 = 0, 
\qquad 
\sigma_0 = 1, 
\qquad 
\sigma = 1.
\]
For this experiment, the transport map $T_\theta(u; y)$ is parameterized as
\[T_\theta(u; y) = u + S_\theta(u; y),
\]
where $S_\theta(u; y)$ is a multilayer perceptron (MLP) 
with the hyperparameters summarized in the first column of Table~\ref{tab:exp_settings} in Appendix~\ref{sec:numerical_details}.

After training, the learned map $T_\theta(\cdot \,; y^\dagger)$ can be used to generate posterior samples for different realizations of observations $y^\dagger$. 
In this example, the posterior distribution is bimodal for $y^\dagger > 0$,
but unimodal for $y^\dagger \le 0.$ The posterior is continuous in $y^\dagger$
\cite{stuart_inverse_2010} and the learned approximate posterior is able to capture the smooth transition from uni- to bimodality. 
In Fig.~\ref{fig:Exp1_post}, we plot the pushforward distribution 
$T_\theta(\cdot \,; y^\dagger)_\# \rho(u)$ and compare it with the true posterior distribution $\pi(u \mid y^\dagger)$.

Since training in this experiment is computationally inexpensive as compared to the other experiments, we use this setting to study the scaling of the empirical estimate of the expected squared error
$\mathbb{E}^{y \sim \kappa}\left[
\DE^2\big(\pi(\cdot \mid y), B_\theta(\cdot \,;y)\big)
\right]$
with respect to the 
size of the training dataset, $N$. The details of the scaling study are presented in Appendix ~\ref{sec:numerical_details}. In summary, we observe that statistical error converges at least at the expected Monte Carlo rate with respect to $N$ for different model sizes. These studies can be used to build a strategy for selecting an appropriate model size that minimizes the  error for each $N$, but we re-emphasize that our main goal in this work is to demonstrate the ability of the new methodology for Bayesian Inference rather than obtaining the optimal convergence rates.



\begin{figure}
    \centering
    \includegraphics[
        width=\linewidth,
    ]{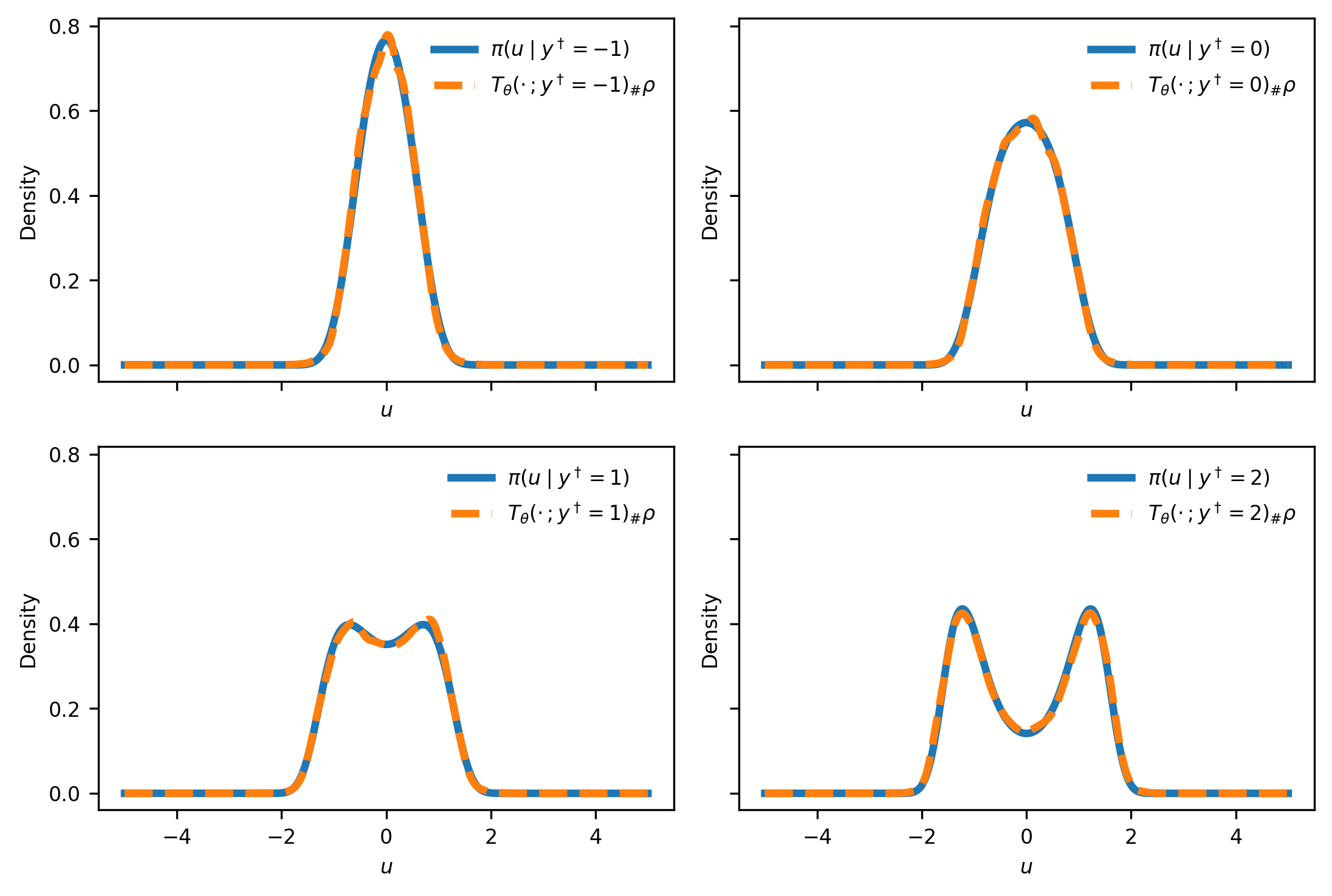}
    \caption{Tractable One-Dimensional Posterior: Close agreement between the true posterior $\pi(u|y^\dagger)$ (blue) and the approximate pushforward distribution $T(\cdot\,;y^\dagger)_\#\rho(u)$ (orange) for different observations $y^\dagger \in \{-1,0,1,2\}$.}
    \label{fig:Exp1_post}
\end{figure}

\subsection{Experiment 2: Darcy Flow Inverse Problem} \label{sec:darcy}

Next, we consider the one-dimensional Darcy flow equation with homogeneous Dirichlet boundary conditions:
\begin{subequations}
\begin{align*}
&- \frac{d}{dx}\left( a(x)\frac{dp}{dx}(x) \right) = 1, \qquad x \in (0,1),\\
&p(0)=0, \qquad p(1)=0.
\end{align*}
\end{subequations}
Here \(p(x)\) denotes the pressure field and \(a(x)>0\) denotes the permeability.
We are interested in the inverse problem of determining $a$, given measurements of $p$. 
However, to enforce positivity of $a$, we parameterize it through the log-permeability field
$u(x)=\log a(x).$
We place a Gaussian prior on the log-permeability field,
\begin{equation*}
u \sim \mathcal N(0,C),
\qquad
C = \sigma^2(-\Delta + \tau^2 I)^{-\alpha},
\end{equation*}
where \(-\Delta\) is the one-dimensional Laplacian on \((0,1)\) with homogeneous
Neumann boundary conditions. As discussed in Subsection~\ref{ssec:transportmap}, the covariance operator is diagonal in the cosine
basis. In particular, after excluding the constant mode, its eigenfunctions and
eigenvalues are given by
\begin{equation*}
\phi_k(x)=\cos(k\pi x),
\qquad
\lambda_k=\sigma^2\big((k\pi)^2+\tau^2\big)^{-\alpha},
\qquad k\in \mathbb N.
\end{equation*}
In this experiment, we use
\begin{equation*}
\tau = 3,
\qquad
\alpha = 2,
\qquad
\sigma = 1.
\end{equation*}
The observation operator consists of pressure measurements at eight equally spaced
interior points \(x_1,\ldots,x_8\). Thus,
\begin{equation}
\label{eq:exp2_obs_op}
 \cG(u)
=
\bigl(
p(x_1;u),\ldots,p(x_8;u)
\bigr)
\in \mathbb R^8.
\end{equation}
The observed data are generated according to Eq.~\eqref{eq:GIPadd}
where the observational noise is Gaussian with the form
\begin{equation*}
\eta \sim \mathcal N(0,\sigma_{\mathrm{obs}}^2 I),
\qquad
\sigma_{\mathrm{obs}} = 1 \times 10^{-3}.
\end{equation*}
The inverse problem is to infer the conditional distribution \(\pi(\cdot \mid y)\)
of the log-permeability field from noisy pressure observations. Since \(u\) is a
function on \((0,1)\), this defines an infinite-dimensional Bayesian inverse
problem. To train the transport map, we generate samples from the joint distribution $\gamma$ over
\((u,y)\). Specifically, we first draw \(u^{(i)} \sim \mathcal N(0,C)\), solve the
Darcy flow equation using a second-order finite-difference discretization on a
uniform grid with \(64\) spatial points, and then generate synthetic observations $y^{(i)}$ using Eq.~\eqref{eq:exp2_obs_op} with additive Gaussian noise. This yields training data pairs
as in Data Assumption \ref{assumption:data}.

Using this data, we minimize the empirical objective in~\eqref{eq:JN} to train a transport map of the form~\eqref{eq:T_CM} where $S_\theta(u,y)$ is a Fourier Neural Operator (FNO). The second column of Table~\ref{tab:exp_settings} summarizes the hyperparameters used for the transport map in this experiment. For each observation \(y\), the learned
map defines the approximate posterior $T_\theta(\cdot \,;y)_\#\rho(u) \approx \pi(u \mid y)$. Since the transport map is trained in an amortized fashion, a single trained model can be used to approximate the posterior for different observation realizations. To illustrate the behavior of the learned posterior across different observations, we draw three representative observations \(y^\dagger\) independently from the marginal observation distribution \(\kappa\). 
Fig.~\ref{fig:Exp2_phys} shows the approximate posterior for these observations.
In each row, the left panel shows the noisy pressure observations, while the right panel shows the corresponding posterior for the log-permeability field \(u(x)\). The true field used to generate the data is shown in red. The learned pushforward posterior \(T_\theta(\cdot \,;y^\dagger)_\#\rho\) is shown in green through its empirical mean and one empirical standard deviation. These quantities are estimated by drawing \(20{,}000\) samples from the reference measure \(\rho\), pushing them forward through \(T_\theta(\cdot \,;y^\dagger)\), and computing the sample mean and standard deviation, and are compared with a reference posterior obtained using pCN, shown in blue.

For each observation \(y^\dagger\), the reference posterior is obtained by running an independent MCMC pCN sampler~\cite{cotter_mcmc_2013} for \(10^6\) steps, after discarding \(2\times 10^5\) initial samples as burn-in. In contrast, the learned transport map is reused without retraining or running a new sampler. Across the three observations, the approximate posterior closely matches the pCN posterior in both mean and standard deviation.

To further compare the learned posterior with the pCN reference, we project the obtained 
posterior samples with either transport or pCN onto the eigenfunctions of the prior covariance operator \(C\).
This provides a mode-by-mode comparison of the marginal posterior distributions
in the Karhunen--Lo\`{e}ve basis. For brevity, we show this comparison for the
observation instance corresponding to the first row of Fig.~\ref{fig:Exp2_phys};
similar behavior is observed for the other observations. Fig.~\ref{fig:Exp2_project2} compares the marginal distributions of the first
six KL coefficients under the prior, the pCN posterior, and the learned pushforward
posterior. The pushforward posterior closely matches the pCN posterior in the
leading modes, where the data have the strongest influence. In higher modes, both
the pCN posterior and the pushforward posterior approach the prior distribution,
indicating that these modes are weakly informed by the observations.

As discussed in Section~\ref{ssec:transportmap}, our main parameterization uses a
Cameron--Martin-informed transport map, in which the neural-operator output is
smoothed by the square root of the prior covariance operator. This construction
encourages the perturbation to lie in the Cameron--Martin space associated with
the prior, and is therefore consistent with the absolute-continuity structure of
the Bayesian posterior. To assess the effect of this structure, we compare it with a baseline
parameterization that uses the same neural-operator architecture but omits the
application of \(C^{1/2}\). This comparison isolates the role of the Cameron--Martin smoothing in the transport map.

To quantify the effect of the Cameron--Martin-informed parameterization, we project
samples from both learned posteriors onto the eigenfunctions of \(C\). For each mode, we compute the one-dimensional Wasserstein-1 distance between the marginal distribution of the learned posterior and the corresponding marginal distribution of the pCN reference posterior. Since these marginals are one-dimensional, the Wasserstein-1 distance is computed in closed form from samples using the empirical quantile functions. The resulting per-mode errors are shown in
Fig.~\ref{fig:Exp2_wass}. The Cameron--Martin-informed map yields consistently smaller errors than the
baseline map, with the improvement most visible in the higher modes. This indicates
that applying \(C^{1/2}\) helps control the small-scale components of the transported
field and improves agreement with the reference posterior in directions that are
weakly informed by the data.

Since the prior is supported on zero-mean fields, posterior samples should likewise
have zero spatial mean. The Cameron--Martin-informed parameterization preserves
this structure by construction. In contrast, the baseline map may generate fields
with a nonzero projection onto the constant mode, resulting in samples that are
inconsistent with the support of the prior.

\begin{figure}
    \centering
    
    \includegraphics[width=0.9\linewidth]{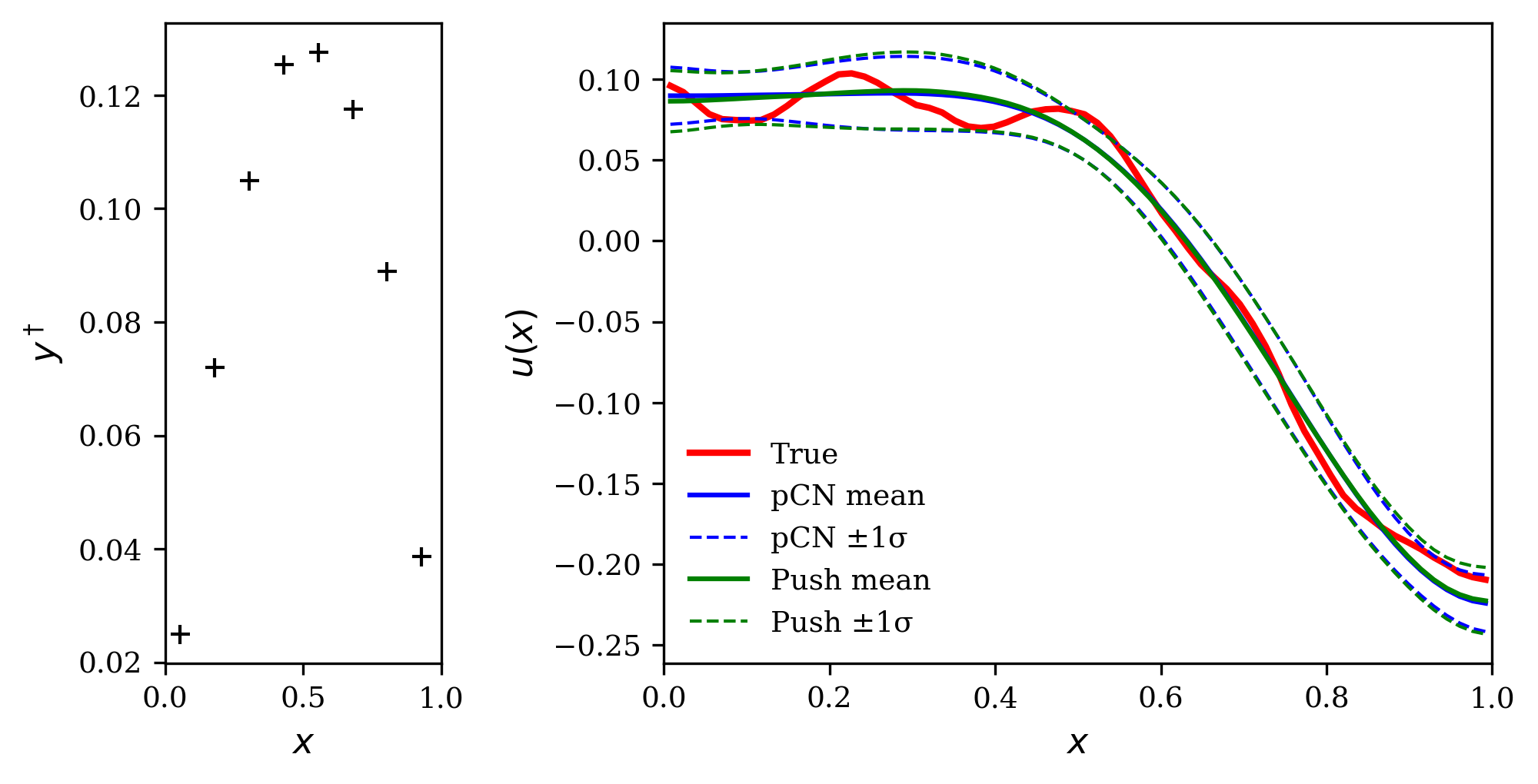}
    
    \vspace{0.5em}
    \includegraphics[width=0.9\linewidth]{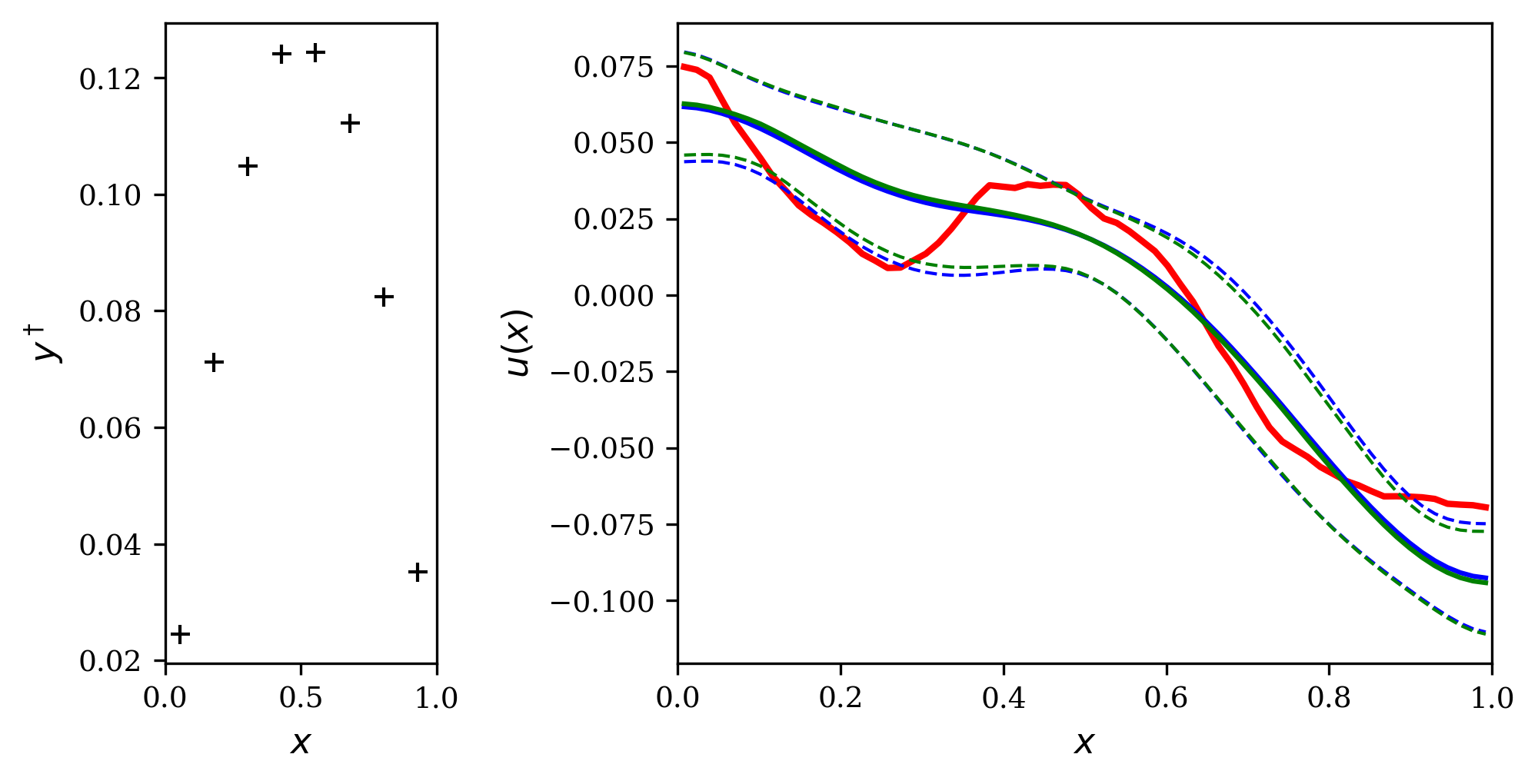}
    
    \vspace{0.5em}
    \includegraphics[width=0.9\linewidth]{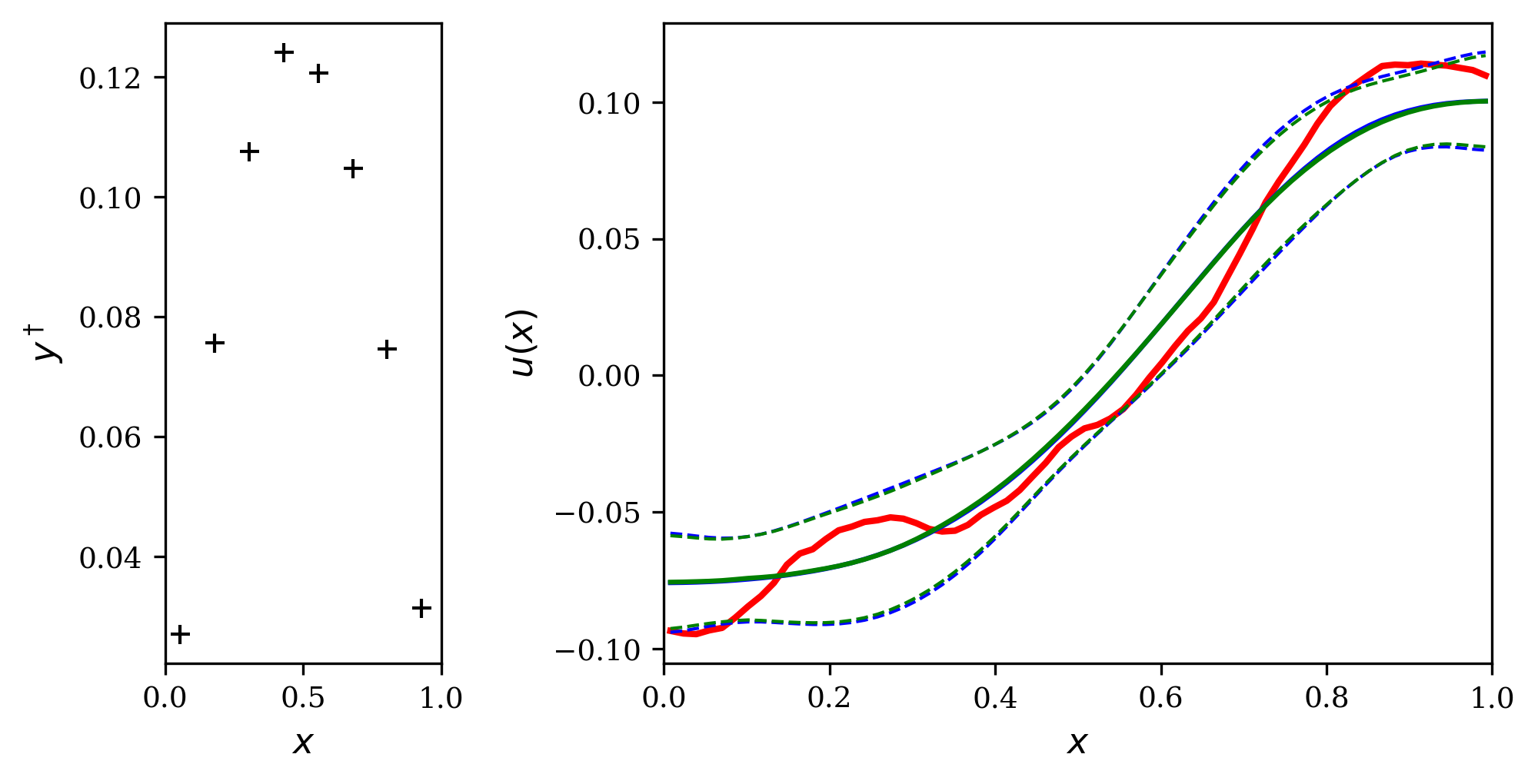}
    \caption{
    Darcy flow inverse problem: Posterior reconstruction of the log-permeability in physical space for three realizations of the noisy observations $y^\dagger$. 
    Each row corresponds to a distinct realization of $y^\dagger$. 
    For each realization, the pCN sampler is initialized and run independently to approximate the posterior, 
    whereas a single learned pushforward map $T(\cdot \,; y^\dagger)_\# \rho$ is reused across all realizations to produce posterior estimates. 
    Within each row, the left panel shows the noisy observations and the right panel shows the reconstruction. 
    The true field $u$ is shown in red, the pushforward approximation in green (mean $\pm$ standard deviation), 
    and the pCN posterior in blue (mean $\pm$ standard deviation). Note the close agreement between the pushforward approximation and the pCN posterior.
    }
    \label{fig:Exp2_phys}
\end{figure}

\begin{figure}
    \centering
    \includegraphics[width=\linewidth]{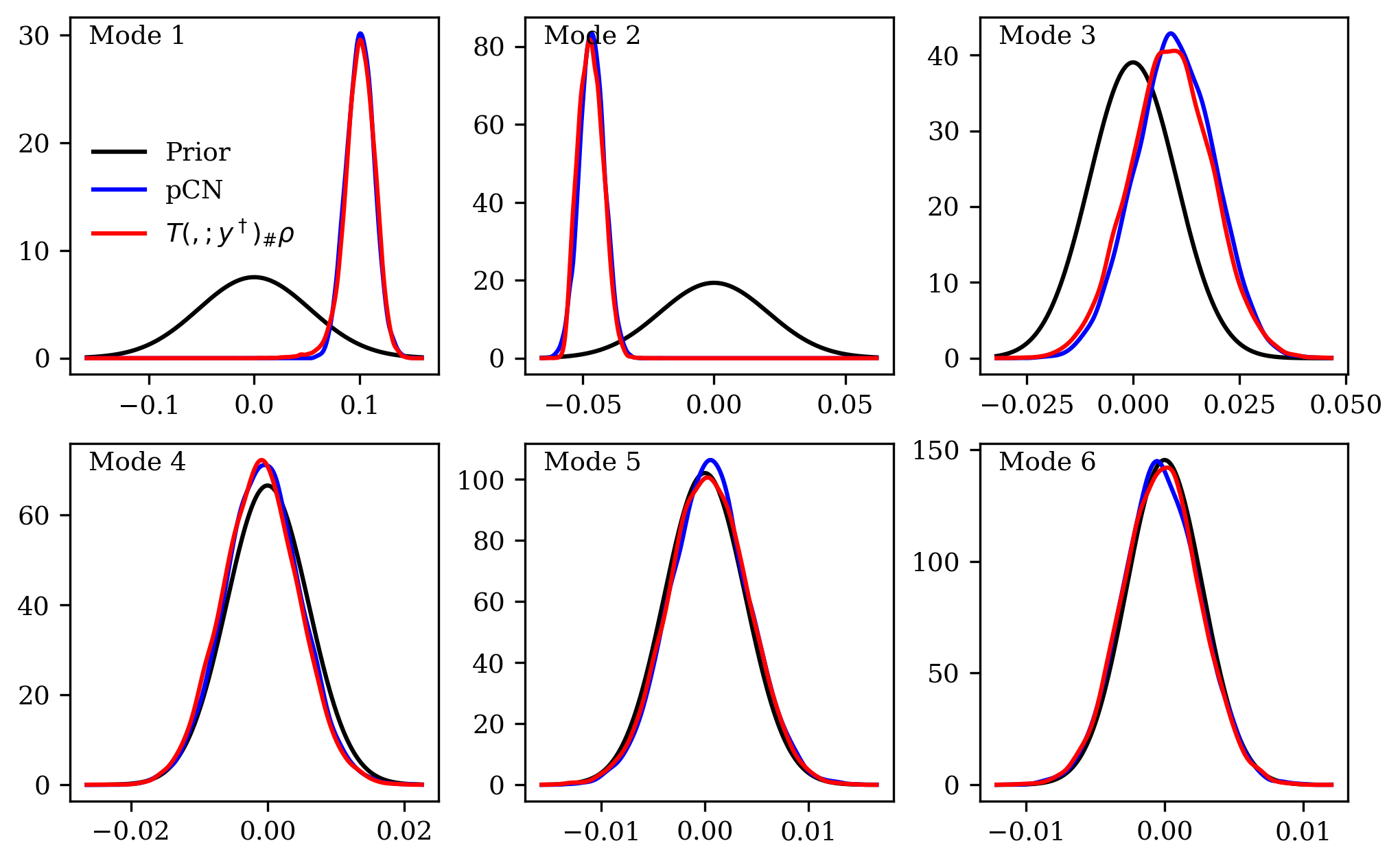}
    \caption{
    Darcy flow inverse problem: Projection onto the KL modes of the posterior distribution. Shown are the first six modes for the pCN posterior (blue), the pushforward $T(\cdot \,; y^\dagger)_\# \rho$ (red) and the prior (black). The pCN and pushforward distributions closely agree in the leading modes, while for higher modes they approach the prior, indicating limited information from the data.
    }
    \label{fig:Exp2_project2}
\end{figure}

\begin{figure}
    \centering
    \includegraphics[width=0.9\linewidth]{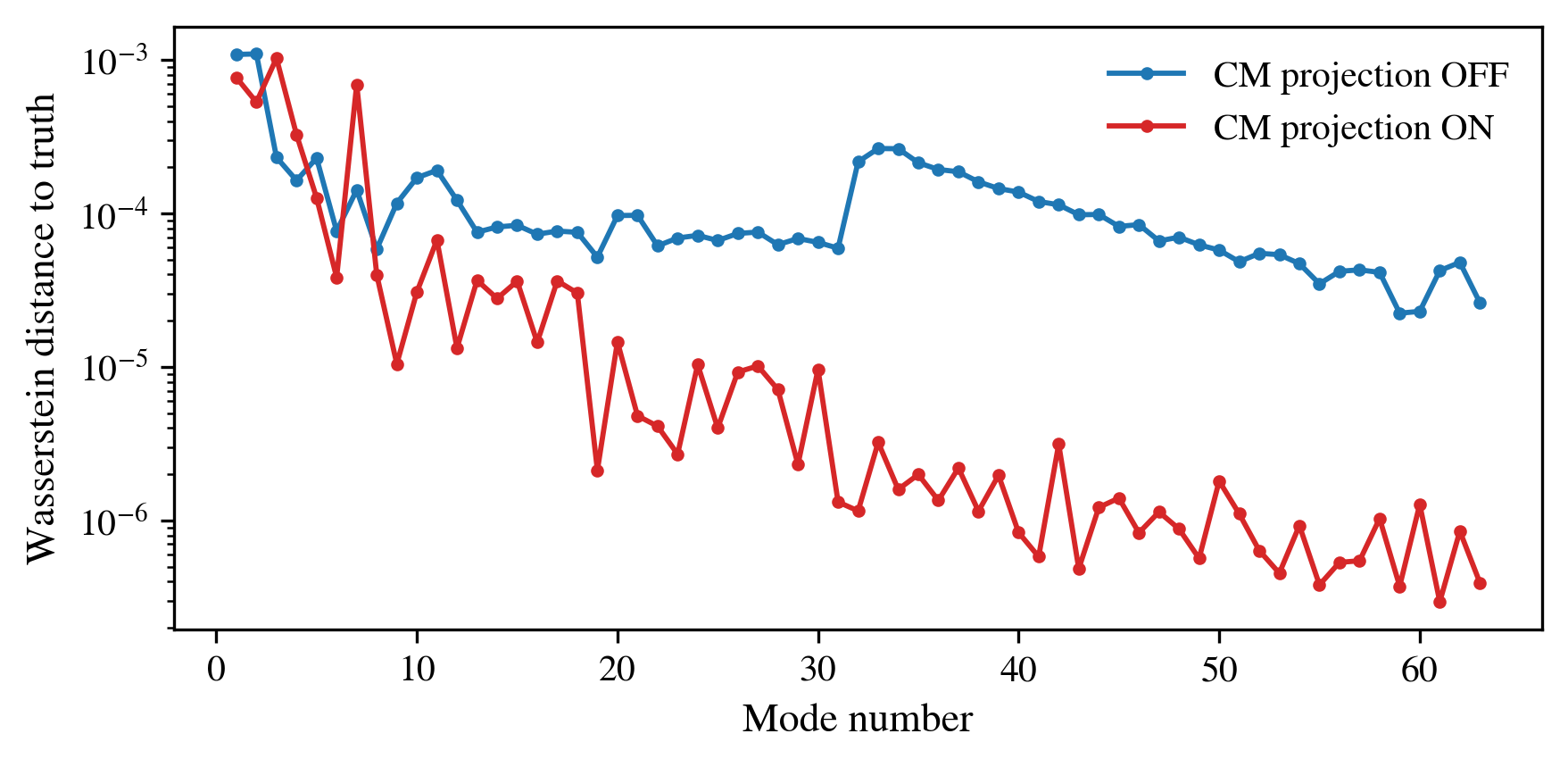}
    \caption{Darcy flow inverse problem: Per-mode Wasserstein error between learned pushforward posterior and the pCN reference posterior. The error is computed between the one-dimensional marginal distributions for the posterior distribution projected onto each KL basis of the prior covariance. The Cameron--Martin-informed map yields smaller errors than the baseline map without \(C^{1/2}\), especially in the higher modes.
}
    \label{fig:Exp2_wass}
\end{figure}

\subsection{Experiment 3: Wave Equation Inverse Problem} \label{sec:waveeq}

As a representative hyperbolic PDE-constrained inverse problem, we consider
wave propagation in a heterogeneous one-dimensional medium. The goal is to infer
the spatially varying wavespeed from indirect observations of the propagating
wavefield. In contrast to the Darcy flow example, the unknown coefficient in this
experiment is piecewise constant and contains sharp interfaces. The forward model is defined by the one-dimensional wave equation
\begin{equation*}
p_{tt}(x,t) - c(x)^2 p_{xx}(x,t) = f(x,t),
\qquad x \in (0,1),
\end{equation*}
with homogeneous Neumann boundary conditions and zero initial conditions. Here, 
\(p(x,t)\) denotes the wavefield, \(c(x)>0\) is the wavespeed, and \(f(x,t)\) is a
localized source term. In this experiment, the source is a Ricker wavelet
localized near the center of the spatial domain. The unknown wavespeed is parameterized through a latent Gaussian field. As in the Darcy flow example, we draw
\[
u \sim \mathcal N(0,C), \qquad C = \sigma^2(-\Delta + \tau^2 I)^{-\alpha},
\]
where \(-\Delta\) is the one-dimensional Laplacian on \((0,1)\) with homogeneous
Neumann boundary conditions. 
The covariance
operator is diagonal in the cosine basis. We exclude the constant mode so that the
latent field has zero spatial mean. In this experiment, we use
\begin{equation*}
\tau = 5,
\qquad
\alpha = 2,
\qquad
\sigma = 10.
\end{equation*}

The physical wavespeed $c$ is obtained from the latent field $u$ through a binary
level-set map
\begin{equation*}
c(x) =
\begin{cases}
c_{\rm high}, & u(x) > 0, \\
c_{\rm low}, & u(x) \le 0,
\end{cases}
\end{equation*}
where
\begin{equation*}
c_{\rm high} = \exp(0.27),
\qquad
c_{\rm low} = \exp(-0.27).
\end{equation*}
This construction produces piecewise-constant wave-speed profiles whose
interfaces are determined by the zero level set of \(u\). Thus, although the prior
is placed on a smooth latent Gaussian field, the physical coefficient entering the
PDE is discontinuous. The observations consist of first-arrival times recorded at a collection of
receiver locations \(x_1,\ldots,x_{N_r}\) \(\in (0,1)\). For each receiver, the arrival time is
defined as the first time at which the magnitude of the recorded signal exceeds a
prescribed fraction of its maximum amplitude. This defines a nonlinear observation
operator
\begin{equation*}
 \cG(u)
=
\bigl(
s_1(u),\ldots,s_{N_r}(u)
\bigr)
\in \mathbb R^{N_r},
\end{equation*}
where \(s_j(u)\) denotes the first-arrival time at receiver \(x_j\). The observed
data are generated according to \eqref{eq:GIPadd},
with additive Gaussian observational noise
\begin{equation*}
\eta \sim \mathcal N(0,\sigma_{\mathrm{obs}}^2 I),\quad \sigma_{\text{obs}}=5\times 10^{-3}.
\end{equation*}
The inverse problem is to infer the conditional distribution
\(\pi(\cdot \mid y)\) of the latent field \(u\), and hence of the induced wavespeed \(c(x)\), from noisy arrival-time observations.

To generate training data, we draw samples \(u^{(i)} \sim \mathcal N(0,C)\), map each
sample to a binary wavespeed \(c^{(i)}(x)\), solve the wave equation using a
finite-difference discretization, and extract first-arrival times at the receiver
locations. Additive Gaussian noise is then applied to the arrival times, producing
joint samples $(u^{(i)},y^{(i)})$ as in Data Assumption \ref{assumption:data}.
These samples are used to train the Cameron--Martin-informed transport map
\eqref{eq:T_CM} with an FNO-based architecture by minimizing the empirical objective \eqref{eq:JN}. The hyperparameters of the transport map are summarized in Table \ref{tab:exp_settings}. Once trained, the
map defines the approximate posterior \(T_\theta(\cdot \,;y)_\#\rho\) for any new
arrival-time observation \(y\).

Fig.~\ref{fig:exp3_spatiotemporal} shows the spatiotemporal evolution of the
wavefield \(p(x,t)\) for the realization corresponding to the first row of
Fig.~\ref{fig:exp3_posterior}. The vertical dashed lines indicate the receiver
locations, where first-arrival times are extracted and used as observations. We then evaluate the learned transport map on three realizations of the noisy
arrival-time observations $y^\dagger$. The results are shown in Fig.~\ref{fig:exp3_posterior}.
In each row, the left panel shows the noisy arrival-time observation, and the
right panel compares the posterior reconstruction of the wavespeed. The true wave
speed used to generate the data is shown in red, the learned pushforward posterior
is shown in green, and the pCN posterior is shown in blue. For each observation realization \(y^\dagger\), the same trained transport map is reused to generate approximate posterior samples, whereas the pCN reference posterior is obtained by running a separate Markov chain. 
The pCN reference posterior is obtained for each $y^\dagger$ by running an independent Markov chain for \(2.5\times 10^6\) steps, after discarding \(1\times 10^5\) samples as burn-in. The learned posterior agrees well with the pCN reference over most of the domain, with larger discrepancies occurring near sharp interfaces in the wavespeed.

We now compare the posterior distributions 
over the latent field \(u\) in the KL basis of the prior covariance. Fig.~\ref{fig:exp3_modes} shows the marginal distributions of
the first six KL coefficients for the observation instance corresponding to the first row of Fig.~\ref{fig:exp3_posterior}. Although the physical wavespeed is obtained through a discontinuous level-set map, the inference is performed on the  Gaussian field \(u\). In this latent representation, the learned pushforward posterior closely matches the pCN reference posterior in the dominant modes.

We also quantify the mode-by-mode discrepancy between the learned posterior and
the pCN reference. Since the prior is placed on the latent field \(u\), we compute
these errors in the KL basis of the prior covariance rather than directly in the
binary wave-speed variable \(c(x)\). For each mode, we compute the Wasserstein
distance between the one-dimensional marginal distribution of the learned
posterior and the corresponding marginal distribution of the pCN posterior. The results are
shown in Fig.~\ref{fig:exp3_wass}. 
We observe that the Cameron--Martin-informed map yields smaller errors than the baseline without
the \(C^{1/2}\) smoothing, with the improvement most visible in the higher modes.
In addition, the baseline map produces a nonzero projection onto the constant
mode, even though this mode is excluded from the prior and the latent field is
restricted to have zero spatial mean. This indicates that the
Cameron--Martin-informed parameterization not only improves agreement with the
pCN reference in the KL coordinates, but also better preserves the support structure of the prior.

\begin{figure}
    \centering
    \includegraphics[width=0.5\linewidth]{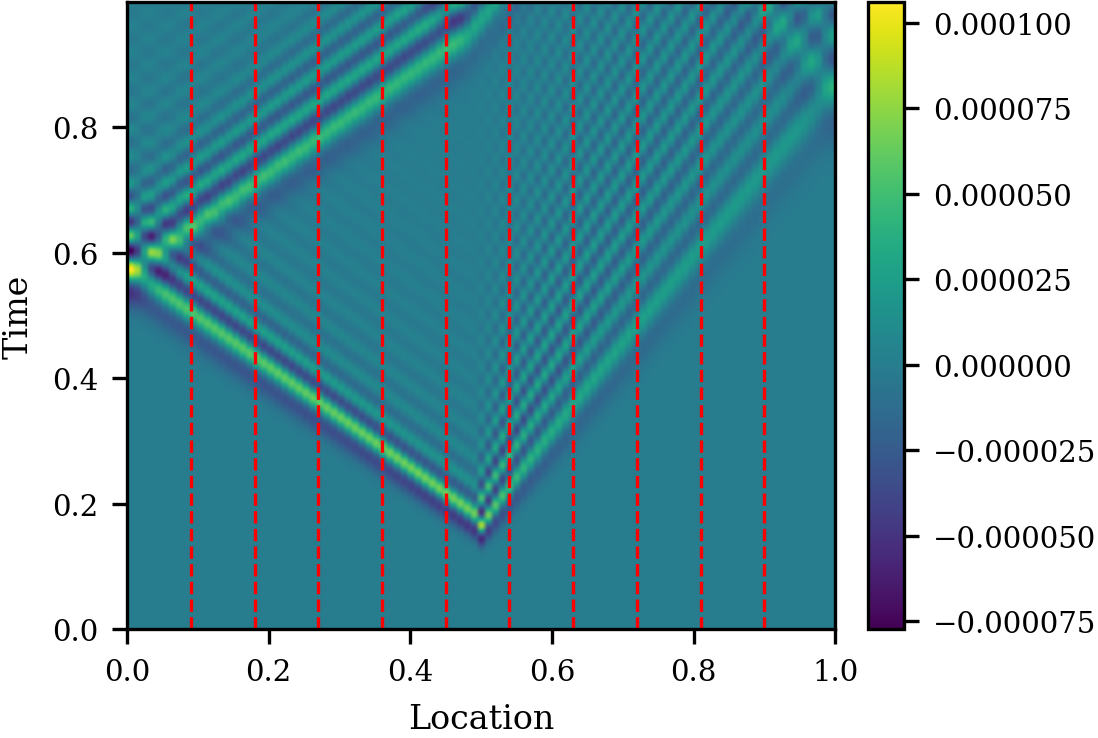}
    \caption{Wave equation inverse problem: Spatiotemporal evolution of the wavefield $p$. Vertical dashed lines denote the location of the receivers $x_j$ corresponding to the observations.}
    \label{fig:exp3_spatiotemporal}
\end{figure}

\begin{figure}
    \centering
    \includegraphics[width=\linewidth]{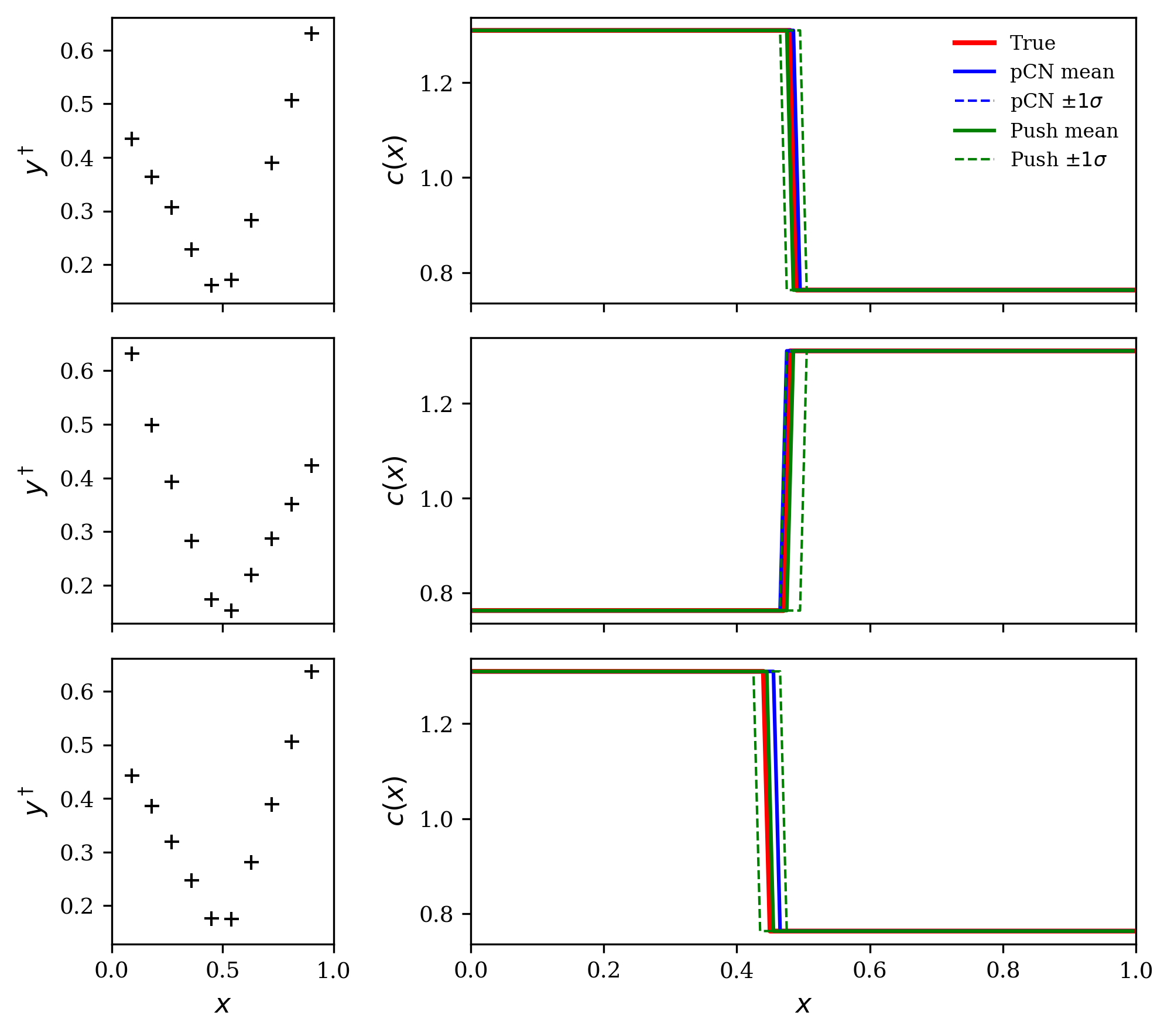}
    \caption{Wave equation inverse problem: Posterior reconstruction across multiple realizations of the
  noisy arrival-time observations $y^\dagger$. Each row corresponds to one realization. The learned pushforward approximation $T(\cdot \,; y^\dagger)_\# \rho$ is compared with an independently run pCN
  posterior. The true field is shown in red. The mean curves show the binary wave speed obtained by thresholding the posterior mean latent field; the dashed bands show the binary fields obtained by thresholding at the $\pm1\sigma$ contours of the latent field, indicating uncertainty in
   the interface location. The pushforward posterior is shown in green and the pCN posterior in blue.
  There is close agreement between the pushforward approximation and the pCN posterior in the mean response, however, the uncertainties differ near the interfaces. 
  }
    \label{fig:exp3_posterior}
\end{figure}

\begin{figure}
    \centering
    \includegraphics[width=\linewidth]{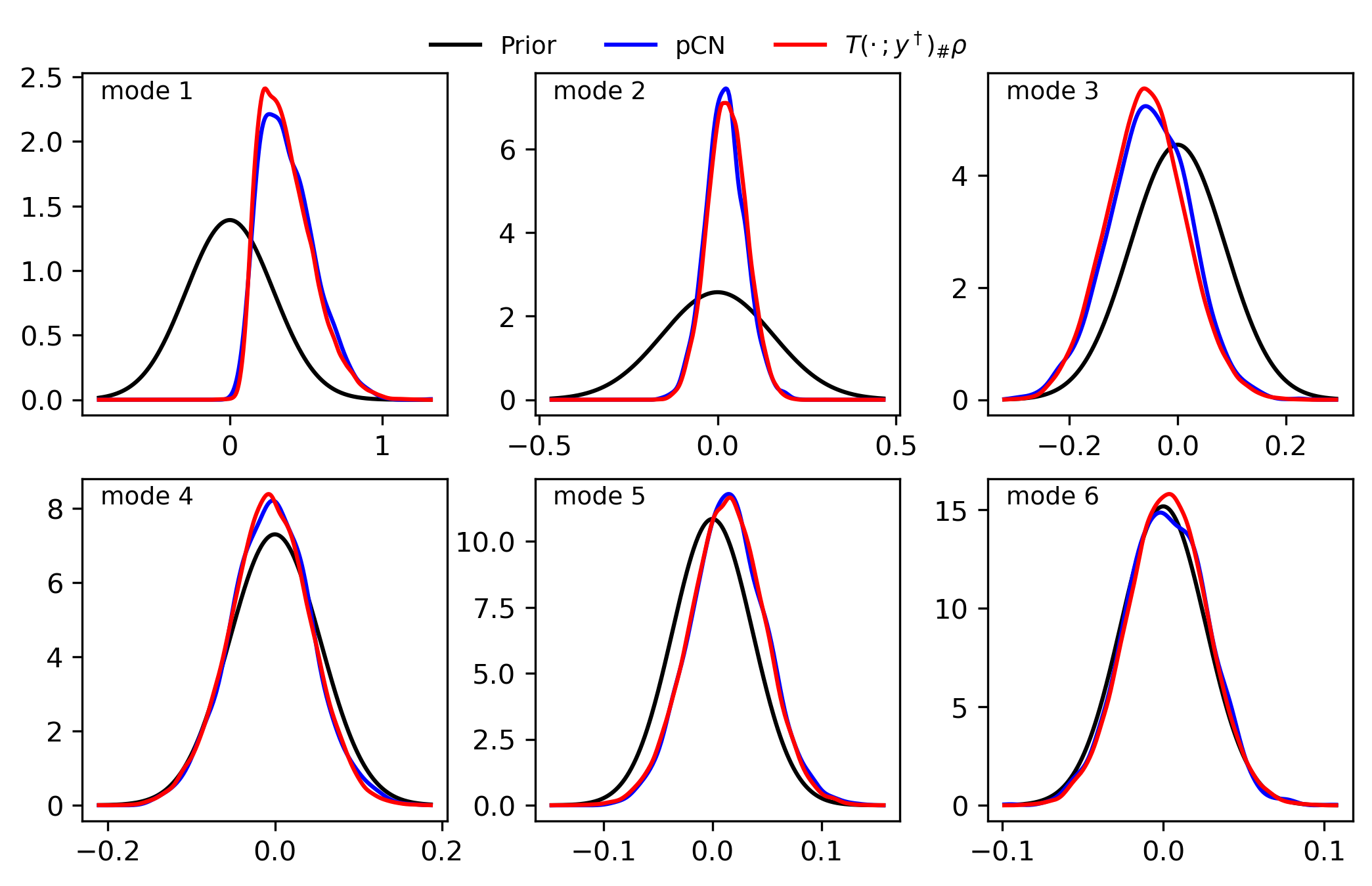}
    \caption{Wave equation inverse problem: Projection of the wave-equation posterior samples onto the first six KL modes. The pCN posterior is shown in blue, the pushforward posterior \(T(\cdot \,; y^\dagger)_\# \rho\) in red, and the prior in black. The close agreement between pCN and the pushforward approximation indicates that the learned map captures the dominant posterior structure.}
    \label{fig:exp3_modes}
\end{figure}

\begin{figure}
    \centering
    \includegraphics[width=0.8\linewidth]{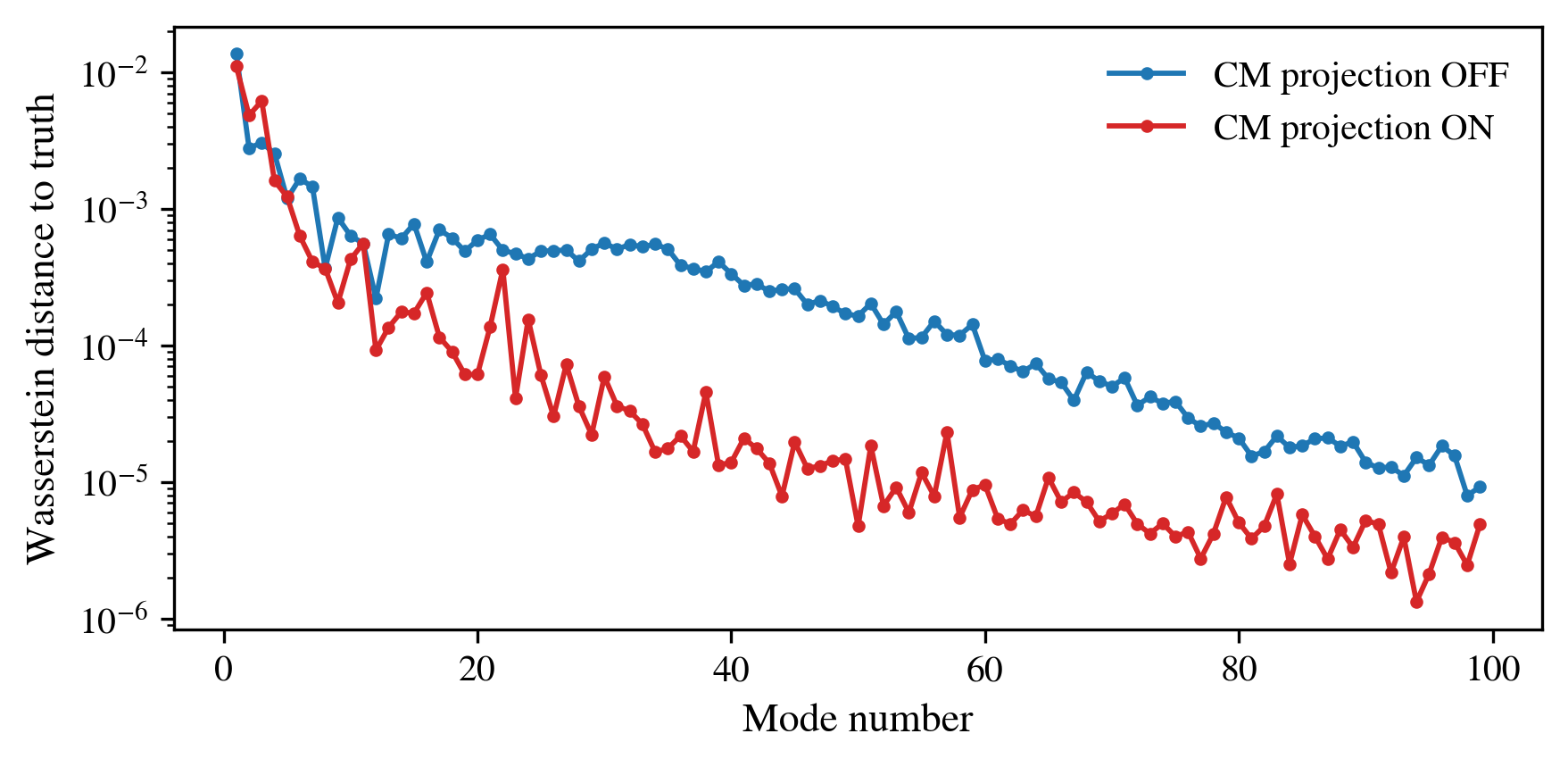}
    \caption{Wave equation inverse problem: Per-mode Wasserstein error for the wave-equation inverse problem, computed in
    the KL basis of the latent Gaussian prior. The Cameron--Martin-informed map gives smaller errors than the baseline without \(C^{1/2}\), especially in higher modes. }
    \label{fig:exp3_wass}
\end{figure}

\section{Conclusion}
\label{sec:conclusion}
We have proposed an amortized transport-based approach for Bayesian inverse
problems; furthermore, we have proposed specific adaptations of the
approach that are relevant when the unknown parameter belongs to an 
infinite-dimensional function space. 
The method learns a conditional transport map that pushes forward
a reference distribution to approximate the posterior distribution
corresponding to a given observation. The map is trained by minimizing 
an averaged, over observations, energy-distance objective, which can be expressed using samples from the joint distribution of parameters and observations. 
This formulation avoids the need to
evaluate likelihoods, construct inverse maps, or compute Jacobian determinants,
and therefore permits flexible neural-operator parameterizations.

A central feature of the proposed construction is the use of
Cameron--Martin-informed transport maps. In the infinite-dimensional setting,
posterior measures arising in Bayesian inverse problems are absolutely continuous
with respect to the prior under standard assumptions. The learned approximation
should therefore preserve this measure-theoretic structure. By parameterizing the
transport map as the identity plus a perturbation in the Cameron--Martin space,
implemented through the action of \(C^{1/2}\), the proposed map is designed to
respect the support and absolute-continuity properties of the prior. The numerical
experiments show that this structure improves agreement with pCN reference
posteriors, particularly in higher KL modes, and prevents the learned map from
introducing components outside the support of the prior, such as the excluded
constant mode in the PDE examples.

The numerical results demonstrate the proposed methodology in application to
a finite-dimensional
nonlinear inverse problem and on two PDE-constrained inverse problems: a Darcy
flow inverse problem and a wave-equation inverse problem. In the finite-dimensional
example, the learned pushforward captures multimodal posterior structure. In the
PDE examples, the learned transport map produces posterior samples that agree well
with pCN reference posteriors in physical space and in the dominant KL modes of
the prior. These results indicate that the proposed energy-distance-based
training objective, together with the Cameron--Martin-informed parameterization,
can provide accurate amortized posterior approximations for nonlinear inverse
problems involving function-valued unknowns.

The main computational benefit of the approach arises in settings where posterior
inference must be performed repeatedly for many different observations. The
training phase requires an upfront cost: one must generate joint samples, solve
the forward model many times, and optimize the parameters of the transport map.
However, this cost is paid only once. After training, posterior samples for a new
observation \(y^\dagger\) are obtained by drawing samples from the prior and
evaluating the learned map \(T_\theta(\cdot \,;y^\dagger)\). This online sampling
procedure is much cheaper than running a new MCMC chain, which would require many
additional forward solves for each new observation. Thus, the amortized approach
is particularly attractive in applications involving repeated inference, such as
uncertainty quantification across many data realizations, real-time or near
real-time inverse problems, and simulation-based experimental design.

\appendix

\section{Gaussian Measure on Hilbert Space} 
\label{sec:preliminaries}
The goal of this section is to introduce the measure-theoretic background needed to formulate transport maps for infinite-dimensional Bayesian inverse problems, with particular emphasis on Gaussian measures, Cameron--Martin spaces, and absolute continuity with respect to the prior. Let \( \mathcal U \) be a separable Hilbert space with inner product \( \langle \cdot, \cdot \rangle_{\mathcal U} \). A \( \mathcal U \)-valued random variable \( u \) is called Gaussian if, for every \( g \in \mathcal U \), the scalar random variable \( \langle g, u \rangle_{\mathcal U} \) is Gaussian in \( \mathbb{R} \). If \( \mathbb{E}[u] = 0 \), we say that \( u \) is centered. The covariance operator \( C : \mathcal U \to \mathcal U \) is defined by  
\[
\langle g, C h \rangle_{\mathcal U}
=
\mathbb{E}\big[ \langle g, u \rangle_{\mathcal U} \, \langle h, u \rangle_{\mathcal U} \big],\quad \text{for all $g,h\in\mathcal U$}.
\]
 Then, \( C \) is a self-adjoint, positive semidefinite, and compact operator. 
By the spectral theorem, there exists an orthonormal basis \( \{e_k\}_{k\ge1} \) of \( \mathcal U \) 
and a sequence of nonnegative eigenvalues \( \{\lambda_k\}_{k\ge1} \) with \( \lambda_k \to 0 \) such that
\( C e_k = \lambda_k e_k \) for all $k$. For $u \in \mathcal{U}$ we write $u_k=\langle u,e_k \rangle_{\mathcal U}$, and similarly for other elements in $\mathcal{U}$. In this basis, a Gaussian random variable \( u \sim \mathcal N(0,C) \) admits the Karhunen--Lo\`{e}ve (KL) expansion \cite{dashti_bayesian_2017}
\[
u = \sum_{k=1}^\infty \sqrt{\lambda_k}\,\xi_k\,e_k,
\qquad \xi_k \sim \mathcal N(0,1)\ \text{i.i.d.}\,.
\]
We denote the law of \( u \) by \( \mathcal N(0, C) \).

Associated with \( \rho = \mathcal N(0, C) \) is a Hilbert space \( \mathcal H \subset \mathcal U \), called the Cameron--Martin space, defined by
\[
\mathcal H := \mathrm{Range}(C^{1/2}).
\]
Equivalently, since \( C^{1/2} \) is a bounded operator on \( \mathcal U \), the space  $\mathcal H$ consists exactly of elements of the form
\[
h = C^{1/2} v, \qquad v \in \mathcal U,
\]
and we may hence define the inner product
\[
\langle h_1, h_2 \rangle_{\mathcal H}
=
\langle C^{-1/2} h_1, \, C^{-1/2} h_2 \rangle_{\mathcal U}.
\]
In the eigenbasis of \( C \), the Cameron--Martin space can be written as
\[
\mathcal H
=
\left\{
h = \sum_{k=1}^\infty h_k e_k :
\sum_{k=1}^\infty \frac{h_k^2}{\lambda_k} < \infty
\right\},
\]
Given two elements in this space: 
\[
h=\sum_{k=1}^\infty h_k e_k,
\qquad
g=\sum_{k=1}^\infty g_k e_k,
\]
the Cameron--Martin inner product is
\[
\langle h,g\rangle_{\mathcal H}
=
\sum_{k=1}^\infty \frac{h_k g_k}{\lambda_k},
\]
and in particular
\[
\|h\|_{\mathcal H}^2
=
\sum_{k=1}^\infty \frac{h_k^2}{\lambda_k}.
\]

Since \( \lambda_k \to 0 \), the condition 
\[
\sum_{k=1}^\infty \frac{h_k^2}{\lambda_k} < \infty
\]
requires the coefficients \( h_k \) to decay faster than those of a typical element of \( \mathcal U \) for which $\{h_k\}_{k \in \mathbb{N}}$ need only be square summable. Indeed, \( \mathcal H \) is compactly embedded into \( \mathcal U \).
Consequently, elements of the Cameron--Martin space exhibit smoother behavior than typical realizations drawn from \( \mathcal N(0,C) \) and a draw \( u \sim \mathcal N(0,C) \) 
belongs to \( \mathcal H \) with probability zero. In the following lemma, we motivate the role of the Cameron--Martin space. Informally, the result shows that if $\rho_h$ is the translation of $\rho$ by $h$, then $\rho_h$ is absolutely continuous with respect to $\rho$ only if the shift $h$ lies in the Cameron--Martin space associated with $\rho$. The proof of the lemma can be found in \cite{bogachev_gaussian_1998}.

\begin{lemma}[Cameron--Martin Theorem]
\label{lemma:CM}
Let $\rho = \mathcal N(0,C)$ be a Gaussian measure on a separable Hilbert space $\mathcal U$. For $h \in \mathcal U$, let $\rho_h$ be the translated measure defined by
\[
\rho_h(A) = \rho(A - h), \qquad A \subset \mathcal U \ \text{Borel}.
\]
Then, $\rho_h$ is absolutely continuous with respect to $\rho$ if and only if $h \in \mathcal H$, where $\mathcal H$ is the Cameron--Martin space associated with $\rho$.
\end{lemma}

Lemma~\ref{lemma:CM} identifies the Cameron--Martin space as the class of shifts that preserve absolute continuity with respect to the Gaussian prior; Lemma~\ref{lem:posterior_ac} then recalls that the Bayesian posterior itself is absolutely continuous with respect to the prior, motivating transport maps whose perturbations lie in the Cameron--Martin space.

\begin{lemma}[Posterior is absolutely continuous with respect to the prior]
\label{lem:posterior_ac}
Let $\rho \in \mathcal P(\mathcal U)$ be the prior, and let $\ell(y^\dagger \mid u)$ be the likelihood for a fixed observation $y^\dagger \in \mathcal Y$. Assume that
\[
Z(y^\dagger)
:=
\int_{\mathcal U} \ell(y^\dagger \mid u)\,\rho(du)
\in (0,\infty).
\]
Then the posterior measure $\pi(\cdot \mid y^\dagger)$ is absolutely continuous with respect to $\rho$, and its Radon-Nikodym derivative is
\[
\frac{d\pi(\cdot \mid y^\dagger)}{d\rho}(u)
=
\frac{\ell(y^\dagger \mid u)}{Z(y^\dagger)}.
\]
In particular, for every Borel set $A \subset \mathcal U$,
\[
\pi(A \mid y^\dagger)
=
\frac{\int_A \ell(y^\dagger \mid u)\,\rho(du)}
{\int_{\mathcal U} \ell(y^\dagger \mid u)\,\rho(du)}.
\]
\end{lemma}
For the proof of this lemma, see \cite[Theorem 14]{dashti_bayesian_2017}.

\section{Proof of Lemma \ref{lemma:objective_e_updated}}
\label{app:proof_objective_e}

\begin{proof}
We begin from the definition of $\LossL(\theta)$:
 \[
\begin{aligned}
\LossL(\theta)
&=\mathbb{E}^{y \sim \kappa}\left[
\DE^2\bigl(\pi(\cdot\mid y),B_\theta(\cdot \,;y)\bigr)
\right]\\
&=\mathbb{E}^{y \sim \kappa}\Big[
2\, \mathbb{E}^{(u,u') \sim \pi(\cdot\mid y) \otimes B_\theta(\cdot \,;y)}
\bigl[\|u - u'\|\bigr]\\
&\qquad\qquad
- \mathbb{E}^{(u,u') \sim \pi(\cdot\mid y) \otimes \pi(\cdot\mid y)}
\bigl[\|u - u'\|\bigr]\\
&\qquad\qquad
- \mathbb{E}^{(u,u') \sim B_\theta(\cdot \,;y) \otimes B_\theta(\cdot \,;y)}
\bigl[\|u - u'\|\bigr]
\Big].
\end{aligned}
\]
For clarity, we write the objective in integral form:
\[
\begin{aligned}  
\LossL(\theta)
&= 2\int_{\mathcal Y}\!\int_{\mathcal U}\!\int_{\mathcal U}
\|u - u'\|\,
\pi(du \mid y')\,
B_\theta(du';y')\,
\kappa(dy') \\
&\qquad - \int_{\mathcal Y}\!\int_{\mathcal U}\!\int_{\mathcal U}
\|u - u'\|\,
\pi(du \mid y')\,
\pi(du' \mid y')\,
\kappa(dy') \\
&\qquad - \int_{\mathcal Y}\!\int_{\mathcal U}\!\int_{\mathcal U}
\|u - u'\|\,
B_\theta(du;y')\,
B_\theta(du';y')\,
\kappa(dy').
\end{aligned}
\]

We simplify the three terms of the objective function separately. The first term can be written as:
\[
2\int_{\mathcal Y}\!\int_{\mathcal U}\!\int_{\mathcal U}
\|u - u'\|\,
B_\theta(du;y')\,
\gamma(du', dy').
\]
Using the pushforward representation
\[
B_\theta(d\z;y') = T_\theta(\cdot \,; y')_\#\muref(d\z),
\]
this becomes
\[
2\int_{\mathcal Y}\!\int_{\mathcal U}\!\int_{\mathcal U}
\|T_\theta(z; y') - u'\|\,
\muref(dz)\,
\gamma(du', dy')\,,
\]
which is
\[
2\mathbb{E}^{\bigl(z, (u^\prime, y^\prime)\bigr)\sim \bigl(\muref \otimes \gamma \bigr)} \| T_\theta(z; y^\prime) - u^\prime\|.
\]

The second term in $\LossL(\theta)$ does not depend on $\theta$. The third term in $\LossL(\theta)$ can be rewritten using the same pushforward representation:
\[
-\int_{\mathcal Y} \int_{\mathcal U} 
\int_{\mathcal U} 
\|T_\theta(z; y'') - T_\theta(z'; y'')\|\,
\muref(dz)\,
\muref(dz^\prime)\,
\kappa(dy''),
\]
which is
\[
- \mathbb{E}^{\bigl(z,z',y''\bigr)\sim \bigl(\muref \otimes \muref \otimes \kappa \bigr)}\| T_\theta(z; y'')-T_\theta(z'; y'') \|
\]

 Since the second term is constant with respect to $\theta$, minimizing $\LossL(\theta)$ is equivalent to minimizing the functional in Eq.~\eqref{eq:J(theta)}.

\end{proof}

\section{Proof of Lemma \ref{lem:ES}}
\label{app:proof_ES}

\begin{proof}
We begin by expanding the definition of the expected squared energy distance into three terms:
\begin{align*}
\mathbb{E}^{y \sim \kappa}
\Big[
\DE^2\bigl(\pi(\cdot \mid y),B_\theta(\cdot \,; y)\bigr)
\Big]
&=
2\int_{\mathcal Y}
\int_{\mathcal U}
\int_{\mathcal U}
\|u-v\|\,
\pi(du\mid y)\,
B_\theta(dv;y)\,
\kappa(dy)
\\
&\qquad
-
\int_{\mathcal Y}
\int_{\mathcal U}
\int_{\mathcal U}
\|u-u'\|\,
\pi(du\mid y)\,
\pi(du'\mid y)\,
\kappa(dy)
\\
&\qquad
-
\int_{\mathcal Y}
\int_{\mathcal U}
\int_{\mathcal U}
\|v-v'\|\,
B_\theta(dv;y)\,
B_\theta(dv';y)\,
\kappa(dy).
\end{align*}
The second term depends only on the true posterior distribution and is independent of
$\theta$. We denote it by
\[
\mathsf{const}
=
\mathbb{E}^{y\sim\kappa}
\Big[
\mathbb{E}^{(u,u')\sim \pi(\cdot\mid y)\otimes \pi(\cdot\mid y)}
\|u-u'\|
\Big].
\]

We now simplify the first and third terms. Using the disintegration identity
\[
\gamma(du,dy)=\pi(du\mid y)\kappa(dy),
\]
the first term can be rewritten as
\begin{align*}
&2\int_{\mathcal Y}
\int_{\mathcal U}
\int_{\mathcal U}
\|u-v\|\,
\pi(du\mid y)\,
B_\theta(dv;y)\,
\kappa(dy)
\\
&
\qquad\qquad =2\int_{\mathcal U\times \mathcal Y}
\int_{\mathcal U}
\|u-v\|\,
B_\theta(dv;y)\,
\gamma(du,dy)
\\
&
\qquad\qquad
=2\,\mathbb{E}^{(u,y)\sim\gamma}
\Big[
\mathbb{E}^{v\sim B_\theta(\cdot \,;y)}
\|u-v\|
\Big].
\end{align*}

Similarly, since the third term depends on \(y\) but not on the outer sample \(u\),
we may insert an expectation over \(u\sim \pi(\cdot\mid y)\) without changing its
value. Equivalently, using again
\(\gamma(du,dy)=\pi(du\mid y)\kappa(dy)\), we obtain
\begin{align*}
&\int_{\mathcal Y}
\int_{\mathcal U}
\int_{\mathcal U}
\|v-v'\|\,
B_\theta(dv;y)\,
B_\theta(dv';y)\,
\kappa(dy)
\\
&\qquad\qquad=
\int_{\mathcal Y}
\int_{\mathcal U}
\int_{\mathcal U}
\int_{\mathcal U}
\|v-v'\|\,
\pi(du\mid y)\,
B_\theta(dv;y)\,
B_\theta(dv';y)\,
\kappa(dy)
\\
&\qquad\qquad=
\int_{\mathcal U\times \mathcal Y}
\int_{\mathcal U}
\int_{\mathcal U}
\|v-v'\|\,
B_\theta(dv;y)\,
B_\theta(dv';y)\,
\gamma(du,dy)
\\
&\qquad\qquad=
\mathbb{E}^{(u,y)\sim\gamma}
\Big[
\mathbb{E}^{(v,v')\sim B_\theta(\cdot \,;y)\otimes B_\theta(\cdot \,;y)}
\|v-v'\|
\Big].
\end{align*}

Substituting these identities into the expansion of the expected squared energy
distance gives
\begin{align*}
&\mathbb{E}^{y \sim \kappa}
\Big[
\DE^2\bigl(\pi(\cdot \mid y),B_\theta(\cdot \,; y)\bigr)
\Big]
\\
&\qquad\qquad=
2\,\mathbb{E}^{(u,y)\sim\gamma}
\Big[
\mathbb{E}^{v\sim B_\theta(\cdot \,;y)}
\|u-v\|
\Big]
\\
&\qquad\qquad\qquad\qquad
-
\mathbb{E}^{(u,y)\sim\gamma}
\Big[
\mathbb{E}^{(v,v')\sim B_\theta(\cdot \,;y)\otimes B_\theta(\cdot \,;y)}
\|v-v'\|
\Big]
-\mathsf{const}.
\end{align*}
We can now group the first two terms inside the same outer expectation:
\begin{align*}
&\mathbb{E}^{y \sim \kappa}
\Big[
\DE^2\bigl(\pi(\cdot \mid y),B_\theta(\cdot \,; y)\bigr)
\Big]
\\
&=
2\,\mathbb{E}^{(u,y)\sim\gamma}
\Bigg[
\mathbb{E}^{v\sim B_\theta(\cdot \,;y)}
\|u-v\|
-
\frac12
\mathbb{E}^{(v,v')\sim B_\theta(\cdot \,;y)\otimes B_\theta(\cdot \,;y)}
\|v-v'\|
\Bigg]
-\mathsf{const}.
\end{align*}
By the definition of the energy score 
$
\ES(\nu,u)
=
\mathbb{E}^{v\sim\nu}\|u-v\|
-
\frac12
\mathbb{E}^{(v,v')\sim\nu\otimes\nu}\|v-v'\|,$
and taking \(\nu=B_\theta(\cdot \,;y)\), we have
\[
\ES\bigl(B_\theta(\cdot \,;y),u\bigr)
=
\mathbb{E}^{v\sim B_\theta(\cdot \,;y)}
\|u-v\|
-
\frac12
\mathbb{E}^{(v,v')\sim B_\theta(\cdot \,;y)\otimes B_\theta(\cdot \,;y)}
\|v-v'\|.
\]
Thus,
\[
\mathbb{E}^{y \sim \kappa}
\Big[
\DE^2\bigl(\pi(\cdot \mid y),B_\theta(\cdot \,; y)\bigr)
\Big]
=
2\,\mathbb{E}^{(u,y)\sim\gamma}
\Big[
\ES\bigl(B_\theta(\cdot \,;y),u\bigr)
\Big]
-\mathsf{const}.
\]
Since \(\mathsf{const}\) is independent of \(\theta\), minimizing the expected energy score is
equivalent to minimizing the objective defined in Eq.~\eqref{eq:objective_e}.
\end{proof}

\section{Additional Details on the Numerical Experiments} 
\label{sec:numerical_details}
In this section, we include the details of the numerical experiments and investigate how the accuracy of the learned transport map scales with the size of the training dataset. Table~\ref{tab:exp_settings} presents the details for the model architecture and training for all three experiments. For the scaling study, we
focus on Experiment 1, which is presented in section \ref{sec:exp1} because, in this case, the true posterior distribution is
tractable and can be accurately approximated by quadrature.


\begin{table}[!htb]
{\footnotesize
\caption{Training and model settings for the numerical experiments.}
\label{tab:exp_settings}
\begin{center}
\begin{tabular}{|c|c|c|c|}
\hline
 & Experiment 1 & Experiment 2 & Experiment 3 \\
\hline
Neural network & MLP & FNO & FNO \\\hline
 Data size& $6\times10^5$& $1\times10^6$ & $1.5\times10^6$\\\hline

Network depth & 8& 5 & 5 \\
\hline
Width & 150& 80 & 100 \\
\hline
Modes (FNO) &--&32&32\\
\hline
Batch size & 1000& 150 & 80 \\
\hline
Epochs & 60& 60 & 15 \\
\hline
Trigger time &-- &20&--\\
\hline
Optimizer & Adam& Adam&Adam \\
\hline
Learning rate &$1\times10^{-3}$ &$1\times10^{-3}$ &$1\times10^{-3}$\\

\hline 

\end{tabular}
\end{center}
}
\end{table}

To measure the accuracy for the scaling study, we compute the expected squared energy
distance between the learned pushforward distribution and the true posterior
distribution, averaged over observations \(y^\dagger \sim \kappa\):
\[
\mathcal E
=
\mathbb E_{y^\dagger \sim \kappa}
\left[
\DE^2\bigl(\pi(\cdot \mid y^\dagger),
B_\theta(\cdot \,; y)\bigr)
\right].
\]
In the main text, we use the prior \(\rho\) as the reference measure and learn a
map of the form
\[
T_\theta(\cdot \,;y)_\# \rho \approx \pi(\cdot \mid y).
\]
In this appendix, we also compare with an alternative construction in which the
reference samples are drawn from the joint distribution \(\gamma\). In this case,
the map takes both the reference pair and the target observation as input and is
trained to push forward samples from the joint distribution toward the posterior
corresponding to \(y^\dagger\). Thus, the scaling study compares two transport
formulations: one using the prior as the reference measure and one using the joint
distribution as the reference measure.

To investigate the scaling with data size, we fix the neural network architecture and vary the number of training samples \(N\). Fig.~\ref{fig:scaling_data} shows the error as a function of \(N\) for several model architectures and for both choices of
reference measure. In both cases, the error decreases approximately as a power law
in the number of training samples, with behavior at least consistent with $\mathcal{O}(N^{-1/2})$.
This is the typical statistical convergence rate associated with Monte
Carlo-type estimators and empirical risk minimization. The observed scaling
therefore suggests that, for the specific fixed model architecture considered here and for the range of $N$ studied here, the dominant error is statistical error from finite training data of size $N.$
\begin{figure}[!htb]
    \centering
    \includegraphics[width=.5\linewidth]{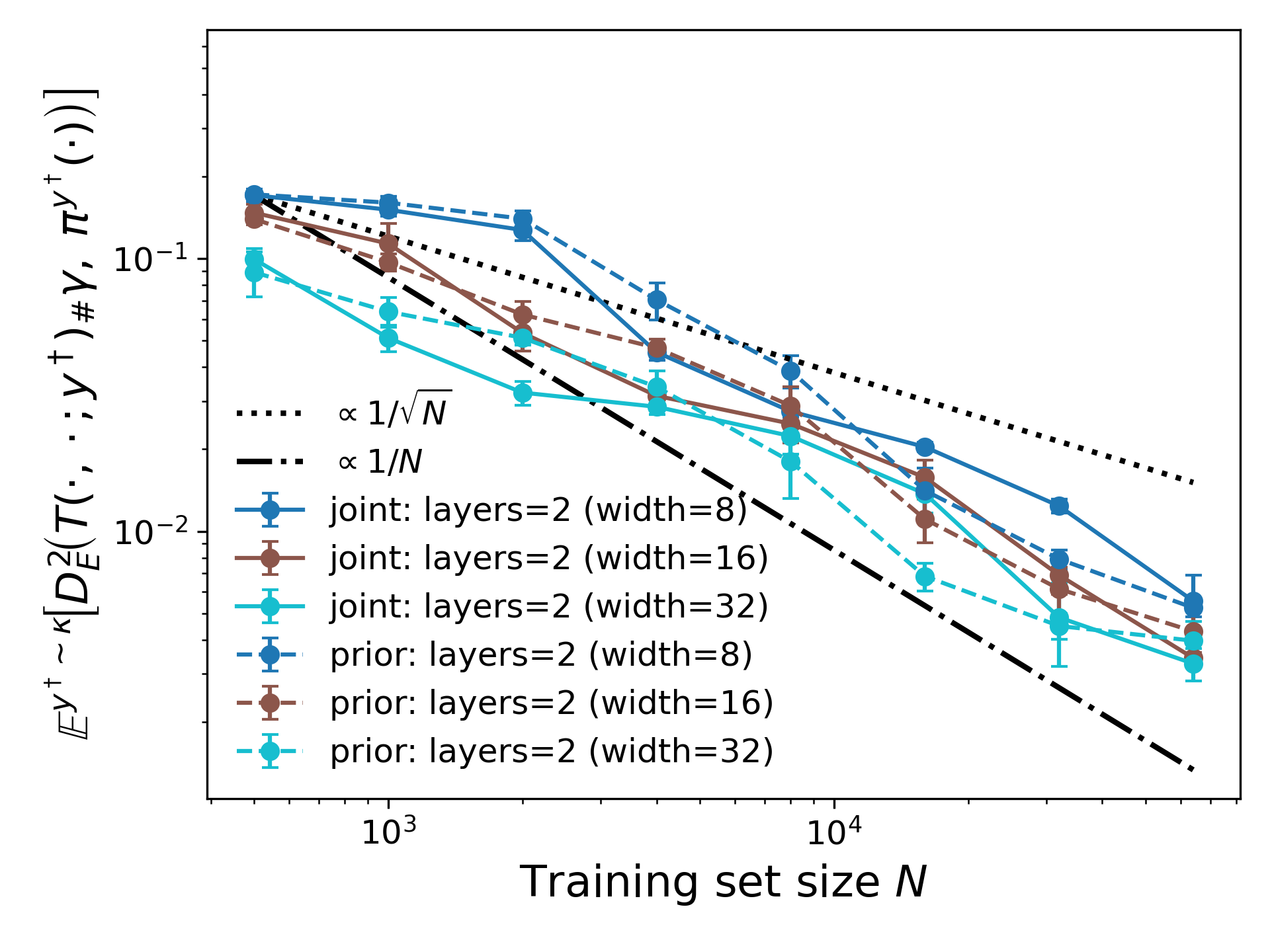}
    \caption{
    Scaling of the validation energy-distance error with the number of training samples \(N\).
    Results are shown for several neural network architectures. The error decreases at least at the Monte Carlo rate \(N^{-1/2}\).
    }
    \label{fig:scaling_data}
\end{figure}

\section*{Reproducibility of computational results}
Code and data that allow readers to reproduce the results in this paper are available at \url{https://github.com/hojjatks/Amortized_energybased_inversion}.
\section*{Acknowledgments}

RB is supported by the NSERC Discovery Grant program. AMS is supported by a Department of Defense (DoD) Vannevar Bush Faculty Fellowship (award N00014-22-1-2790), which also supports HK.

\section*{Author Contributions}
The authors in this paper are listed in alphabetical order. 
\textbf{Ricardo Baptista}: Conceptualization, Methodology, Formal Analysis, and
Writing--Review \& Editing.
\textbf{Hojjat Kaveh}: Conceptualization, Methodology, Software, Investigation, Validation, Formal Analysis, Visualization, Writing--Original Draft, Writing--Review \& Editing.
\textbf{Andrew Stuart}: Conceptualization, Methodology,
Formal Analysis, Supervision, and Writing--Review \& Editing. 

\bibliographystyle{siamplain}
\bibliography{ref_new}
\end{document}